\documentclass[3p]{elsarticle}

\usepackage{lineno,hyperref}

\modulolinenumbers[5]

\journal{arXiv}

\usepackage{amsfonts}
\usepackage{amsmath, amscd,amssymb,bm,amsbsy,epsf}
    \usepackage{enumerate}
    \usepackage{paralist}
\usepackage{bbm, dsfont}
\usepackage{graphicx,color}
\usepackage{subfigure}
\usepackage{verbatim}
\usepackage{inputenc}
\usepackage{multicol}
\usepackage{balance}
\usepackage{verbatim}
\usepackage{bbm, doi}
\usepackage{accents, booktabs}

\usepackage{sfmath} 
\usepackage{sans} 

\usepackage{multirow}

\newtheorem{theorem}{Theorem}

\newdefinition{definition}{Definition}
\newdefinition{hypothesis}{Hypothesis}
\newdefinition{remark}{Remark}
\newproof{proof}{Proof}
\newdefinition{example}{Example}

\DeclareMathAlphabet{\mathpzc}{OT1}{pzc}{m}{it}

\DeclareMathOperator*{\col}{col}
\DeclareMathOperator*{\Numb}{Index}

\DeclareMathOperator*{\cl}{cl}

\def\u{\mathbf{u} }
\def\v{\mathbf{v} }
\def\w{\mathbf{w} }
\def\x{\mathbf{x} }
\def\y{\mathbf{y} }
\def\m{\mathbf{m} }

\def\iversor{\boldsymbol{\widehat{\imath} } }
\def\kversor{\boldsymbol {\widehat{k} } }
\def\jversor{\boldsymbol {\widehat{\jmath} } }

\def\zhat{{\boldsymbol {\widehat{z} } } }

\def\n{\boldsymbol{\widehat{n} } }

\def\Proj{\text{Proj}}

\def\E{\mathcal{E} }
\def\triang{\mathcal{T}}
\def\graph{\mathcal{G}}
\def\Kz{K^{\zeta}}
\def\Gammatilde{\widetilde{\Gamma}_{ \Omega, \triang}}

\def\nuhat{\widehat{\nu } }

\def\nutilde{ \boldsymbol{\widehat{ \nu} } }

\def\curv{\textbf{curv}}
\def\grav{\textbf{grav}}
\def\g{g}
\def\fric{\textbf{fric}}

\def\prob{\boldsymbol{\mathbbm{P}}}
\def\ind{\boldsymbol {\mathbbm{1} } }

\def\R{\boldsymbol{\mathbbm{R} } }

\def\defining{\overset{\textbf{def}}=}
\def\V{\boldsymbol{ V } }
\def\H{\boldsymbol{ H } }
\def\grad{\textbf{grad } }
\def\div{\textbf{div } }
\def\itfc{\textbf{itfc} }
\def\bdry{\textbf{bdry} }
\def\ext{\textbf{ext} }
\def\ones{\mathbf{ 1 } }

\begin{document}

\begin{frontmatter}

\title{Tangential Fluid flow within 3D narrow fissures: Conservative velocity fields \\
on associated triangulations and transport processes.}
\tnotetext[mytitlenote]{This material is based upon work supported by the project HERMES 27798 from Universidad Nacional de Colombia, Sede Medell\'in..}


\author[mymainaddress]{Fernando A Morales}

\cortext[mycorrespondingauthor]{Corresponding Author}
\ead{famoralesj@unal.edu.co}

\address[mymainaddress]{Escuela de Matem\'aticas
Universidad Nacional de Colombia, Sede Medell\'in \\
Calle 59 A No 63-20 - Bloque 43, of 106,
Medell\'in - Colombia}

\author[mymainaddress]{Jorge M Ram\'irez}

\begin{abstract}
For a fissured medium with uncertainty in the knowledge of fractures' geometry, a conservative tangential flow field is constructed, which is consistent with the physics of stationary fluid flow in porous media and an interpolated geometry of the cracks. The flow field permits computing preferential fluid flow directions of the medium, rates of mechanical energy dissipations and a stochastic matrix modeling stream lines and fluid mass transportation, for the analysis of solute/contaminant mass advection-diffusion as well as drainage times. 
\end{abstract}

\end{frontmatter}



%
%

%

\section{Introduction}

Fissured media are common geological structures. In such a medium, fast flow occurs on the cracks while the rock matrix constitutes the slow flow region. The associated transport phenomenon has been extensively studied from several points of view and at different scales of modeling due to its remarkable importance in different fields such as: oil extraction, water supply, pollution of subsurface streams and soils, waste management, etc. Several models of coupled systems of partial differential equations have been proposed, such as double \cite{ArbogastDouglasHornung, Barenblatt} and multiple porosity systems \cite{SpagnuoloWright}, microstructure models \cite{Show97} and the coupling of laws at different scales: see \cite{ArbLehr2006} for an analytic approach of a Darcy-Stokes system and see \cite{LesinigoAngeloQuarteroni} for a numerical treatment of a Darcy-Brinkman system. In addition, various numerical \cite{ArbBrunson2007, Mclaren} or numerical/analytical \cite{ShowalterPeszynska, WangFengHan} works have dealt with the discretization and numerical aspects of the proposed models.

For small values of the Reynolds number, the saturated flow within the fissures is predominantly parallel to the surface hosting it. This fact has been mathematically shown in \cite{ArbogastDouglasHornung, WalkingtonShowalter} using homogenization techniques and in \cite{SpagnuoloCoffield, Morales, MoralesShow2} via asymptotic analysis. 
Assuming further that fissures are thin enough as to forgo any transverse variation of velocities, and that the surrounding rock matrix is impermeable, two degrees of freedom at each point remain to be resolved for a complete identification of the velocity field. Another important aspect is that surveys of fissured media at the field scale are usually performed at discrete (albeit numerous) sampling points (see \cite{Tassone} for an imaging method). The three-dimensional geometry of the fissure is then usually modeled as a triangulated surface with nodes located on points of known coordinates. For physical processes that depend on the fissure's geometry like fluid flow, it is therefore desirable to construct models that yield approximate solutions at the resolution of the triangulated surface. 

The considerations laid out in the previous paragraph encompass the main motivation for the present work. Our starting point is the Darcy's equation for the two-dimensional flow velocity and pressure on the region of the X-Y plane limited by the fissure. This approach is hence limited to fissures whose three-dimensional triangulated surface can be orthogonally projected onto a triangulated region of the X-Y plane in a biunivoque way. The flow is driven by the gravitational potential and a prescribed external pressure potential along the domain. We then compute a conservative two-dimensional velocity field that is piece-wise constant on the triangulation elements, and that best approximates the solution to the discretized Darcy's equation. By conservative in this context, we mean a velocity field for which the net discharge along each element boundary vanishes. The final step in the construction is a lifting operation, again conservative, yielding the corresponding three-dimensional velocity field everywhere parallel to the triangulated surface.

The main feature of the proposed model is its simplicity. Through the action of linear operators whose dimension is comparable to the number of elements in the triangulation, our methodology yields approximate streamlines which are parallel among them within each triangle, and everywhere parallel to the triangulated surface. 

We present two applications of the proposed model aimed at quantifying the fissure's geometry effect on flow and transport. First we consider the problem of dispersion of a tracer that is being advected within the fissure by the constructed velocity field. The evolution of the concentration of the tracer can be characterized in discrete time by a linear operator which is explicitly derived here. As a second application, we consider different mechanisms of energy dissipation. Functionals to estimate dissipation rates associated to curvature, friction and gravity are explicitly derived and computed for fissures of different exemplifying geometries. 

We close this section introducing the notation. Vectors in $\R^{\!3}$ are denoted with bold characters  $ \x = (x_1, x_2, x_3) \in \R^3 $, $\vert \x \vert$ indicates the Euclidean norm of $ \x $, $ \x^T $ its transpose, and $\widetilde{x} = (x_{1}, x_{2})$ its projection on the first two coordinates. Projections along vectors or onto subspaces are denoted by the operator $ \Proj $, for instance, if $ \u \in \R^d $,
\begin{equation}\label{Definition projection}
\Proj_{\x}(\u) \defining (\u \cdot \x) \frac{\x}{ \Vert \x \Vert^2} 
\end{equation}
denotes the projection of $ \u $ along the direction of $ \x $. For a given subset $A$ of $ \R^d $, $ d=1,2,3 $,  we denote its Lebesgue measure by $\vert A \vert$, its cardinal by $\#A$, its closure by $\cl(A)$ and its boundary by $ \partial A $. The symbol $ \ind_A $ denotes the indicator function of the set $ A $: $ \ind_A(\x) =1 $ if $ \x \in A $, zero otherwise. 

\section{The triangulated fissure and the class of element-wise constant vector fields}

Our approach rests on the assumption that the fissure's physical surface can be approximated by a piecewise linear affine triangulation constructed from discrete sampled points. Below, we start delimiting the type of fissures to be analyzed.

\begin{definition}\label{Def sample domain and T}
Let $S = \{\widetilde{x}_{1}, \widetilde{x}_{2}, \ldots, \widetilde{x}_{n}\}$ be the set of distinct points in $ \R^2 $, such that for all $ i=1,\dots,n $, the three-dimensional location $ \big(\widetilde{x}_i,\zeta(\widetilde{x}_i) \big) $ of the fissure surface is known; $ S $ will be called the \textbf{sample set}. The \textbf{sample domain} $ \Omega \subset \R^2$ is the convex hull of $ S $. 
\end{definition}

\begin{hypothesis}\label{Hyp Fissure Geometric Conditions}
	We assume that the \textbf{fissure} is a surface $\Gamma = \big\{ \big(\widetilde{x}, \zeta(\widetilde{x}) \big): \widetilde{x}\in G \big\}$ defined by a continuous function $\zeta:G\rightarrow \R$ on a convex domain $G$ of $\R^{2}$ with $ \Omega \subseteq G $. In particular the continuous 2-D manifold $\Gamma$ has an Atlas containing one element.
\end{hypothesis}
The sample domain $ \Omega \subset \R^2$  will be partitioned with a \textbf{triangulation} $ \triang $ constructed as follows. See Figure \ref{Fig 3DCrack}. 
\begin{definition}\label{Def admissible triangular mesh}
The \textbf{triangular mesh} $ \triang $ is a collection of open triangular disjoint subsets of $\Omega$ such that 
if any two triangles $ K,L \in \triang $ satisfy $\vert \partial K \cap \partial L \vert > 0$ then, $\partial K \cap \partial L$ is a common edge. We also introduce the following notation associated to the triangulation $\triang$:  
\begin{enumerate}[(i)]
\item Denote by $\E$ the \textbf{edges} of the triangulation; $\E_{\bdry}$ are the edges of the triangulation contained in the boundary $\partial \Omega$ of the domain and the edges $\E_{\itfc} \defining \E - \E_{\bdry}$ are the interfaces between pairs of elements in $\triang$.

\item Edges in $ \E_{\bdry} $ are regarded as ``boundary elements'' of an extended triangulation $ \triang_{\ext} \defining \triang \cup \E_{\bdry} $. 

\item Any edge $e \in  \E $ is common to two elements, namely $ K \in \triang, L \in \triang_{\ext}$ and it is thus denoted as $ e = K|L = L|K $. Its length is $ \sigma_e \defining |e| $.
  
\item Given an element $ K \in \triang$ and one of its edges $ e = K|L $, $ L \in \triang_{\ext}$, we denote by $ \nuhat_{K | L} $ the unitary vector that is normal to $ e $ and points outwards $ K $. See Figure \ref{Fig 3DCrack}.
\end{enumerate}

\end{definition}
\begin{figure}\centering
\includegraphics[scale = 1]{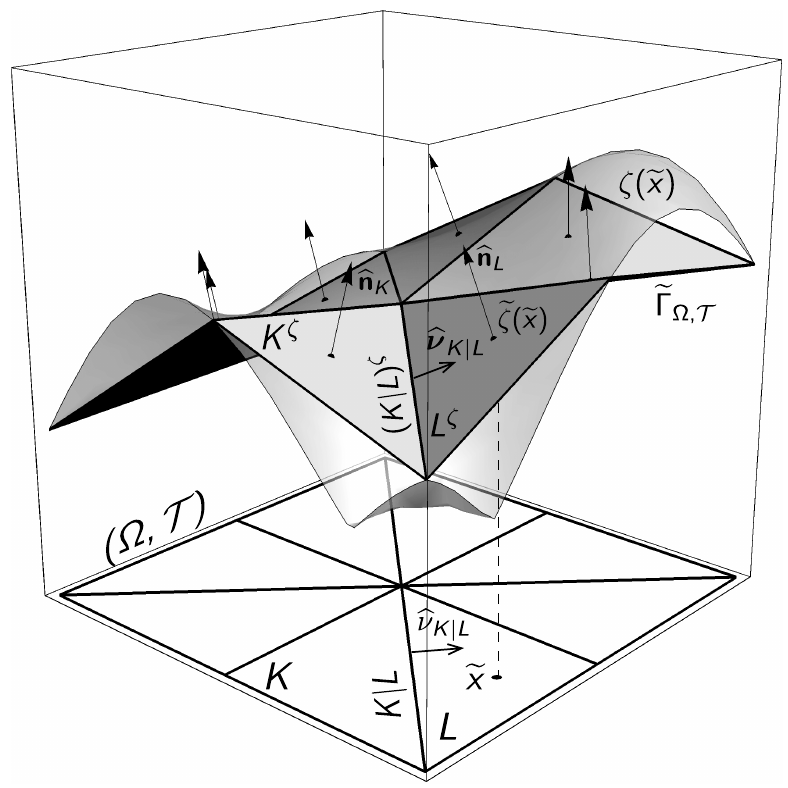}
\caption{Schematic representation of a fissure, a triangulation and its associated elements described in Definitions \ref{Def admissible triangular mesh} and \ref{Def lifting} \label{Fig 3DCrack}}
\end{figure}

In particular, the numerical experiments in Section \ref{Sec Numerical Experiments}, will assume that $\triang$ is a Delaunay triangulation of the sample set $ S $. We now introduce the notation for the associated three-dimensional triangular elements.
\begin{definition}\label{Def lifting}
Let $\Omega$, $ \triang $ be as in Definitions \ref{Def sample domain and T} and \ref{Def admissible triangular mesh}, 
\begin{enumerate}[(i)]
\item Let $K \subseteq \Omega \subset \R^2$ be a triangular element of $ \triang $ with vertices $\{z_{\,\ell}: 1\leq \ell\leq 3\}$. We define the $\Kz$ as the triangle formed by the points $\{(z_{\,\ell}, \zeta(z_{\,\ell})):1\leq \ell\leq 3\}\subset \R^{3}$.
\item For any point $ \widetilde{x}\in \cl(K) $, we define $ \widetilde{\zeta}(\widetilde{x}) $ as the unique point $ \x $ in $\Kz$ such that its horizontal projection agrees with $\widetilde{x}$, namely
$ \x - ( \x \cdot\kversor ) \, \kversor = \widetilde{x} $. In particular, the function $ \widetilde{\zeta} $ agrees with $ \zeta $ on the sample set $ S $.
\item For an edge $ e \in \E $ we define $ e^{\zeta} $ in the analogous way and denote its length by $ \sigma^{\zeta}_{e} $.
\item 
   The \emph{triangulation} of $\Gamma$ relative to $\Omega$ and $\triang$, is given by the collection of the previously defined triangles; it is denoted by
\begin{equation}\label{Def triangulation surface P, T}
 \widetilde \Gamma_{ \Omega, \triang}\defining\bigcup\left\{\Kz: K\in \triang\right\} .
\end{equation}
\item The vector $\n_{K}$ indicates the upwards unitary vector, normal to the surface of the element $\Kz$.
\item For each $K\in \triang, L \in \triang_{ext}$, $\nutilde_{K|L}$ denotes the vector that is normal to both $ \n_{K} $ and the edge $ (K|L)^{\zeta} $, pointing outwards $ \Kz $. 
\end{enumerate}
\end{definition}
We will throughly use vector fields of two distinct types: fields defined over the sample domain $ \Omega $ and with values on $ \R^2 $; the others defined over the triangulated surface $ \Gammatilde $ and with values on $ \R^3 $. Moreover, we will specialize on those vector fields that are constant over the triangles that compose their corresponding domains of definition. The space of  \textbf{element wise constant} flow fields on $ \Omega $ is denoted by
\begin{equation}\label{Eqn Element Wise Constant Space}
\H(\Omega, \triang) \defining\Big\{
\v(\widetilde{x}) = \sum_{K \, \in \,\triang} \v_{K} \, \ind_{K}(\widetilde{x}) : 
\v_{K} \in \R^{2}, \; \text{for all } \, K\in \triang
\Big\},
\end{equation}
and in particular, a vector field in $ \H(\Omega, \triang) $ will be called \textbf{conservative} if it belongs to the space $\V(\Omega, \triang)$ defined as follows,
\begin{equation}\label{Eqn Conservative Space}
	\V(\Omega, \triang) \defining \left\{
	\v  \in \H(\Omega, \triang) \; \text{ such that if }  \vert \partial K \cap \partial L \vert > 0 ,\; \text{then } 
	\v_{K}\cdot \nuhat_{K|L}  =  
	- \v_{L}\cdot \nuhat_{L|K} \right\}.
\end{equation}
Observe that $ \V(\Omega, \triang) $ is a special case of the Raviart-Thomas finite element space (see \cite{BraessFEM}). We also define the space of \textbf{conservative potentials} $E(\Omega, \triang)$ by
\begin{equation}\label{Eqn Conservative Potential Space}
E(\Omega, \triang) \defining \big\{
q \in P_{1}(\Omega, \triang) : 
\grad q \in \V(\Omega, \triang)
\big\} .
\end{equation}
%
%
Equivalently, if $\grad: P_{1}(\Omega, \triang)\rightarrow \H(\Omega, \triang)$, then $E(\Omega, \triang) \equiv \grad^{-1}\big(\V(\Omega, \triang)\big)$. 

Among three-dimensional vector fields defined on the triangulated surface $ \Gammatilde $, we will be interested on element-wise constant vector fields that are everywhere parallel to the surface. We introduce the space
\begin{equation}
\widetilde{\H}(\Omega,\triang) \defining\Big\{
\w(\x) = \sum_{K \, \in \,\triang} \w_{K} \, \ind_{K}(\x)  : 
\w_{K} \in \R^3, \; \w_{K} \cdot \n(K) = 0 \text{  for all } \, K\in \triang, \; \x \in \Omega \Big\}, 
\end{equation}
as well as the corresponding space of conservative vector fields
\begin{equation}\label{Eqn Conservative 3D Space}
	\widetilde{\V}(\Omega, \triang) \defining \left\{
	\w  \in \widetilde{\H}(\Omega, \triang) \; \text{ such that if }  \vert \partial K \cap \partial L \vert > 0 ,\; \text{then } 
	\w_{K}\cdot \nutilde_{K|L}  =  
	- \w_{L}\cdot \nutilde_{L|K} 
	\right\} .
\end{equation}

%
\section{Construction of the velocity field}\label{Sec basic velocity field}
%
%
In this section we construct a family of conservative velocity fields hosted on the triangulated surface $ \widetilde\Gamma_{\Omega, \triang} $ and defined by the surface's geometry as well as the pressure potential. This is accomplished in two steps. First, a two-dimensional \textbf{master conservative vector field} $ \v \in \V(\Omega, \triang) $ is calculated from the equations of fluid flow in porous media. Second, the vector field is lifted to the surface $ \Gammatilde $ in such a way that the resulting \textbf{global velocity field} $ \u $ lands in the space $ \widetilde{\V}(\Omega, \triang) $.

Our starting poing is the saturated porous media flow equation on $ \Omega $: 

\begin{subequations}\label{Eq Porous Media Equation}
	\begin{equation}\label{Eq Darcy Equation}
	a\, \v +  \, \grad p = - \g \, \rho \,   \grad \zeta - \grad P , 
	\end{equation}
	\begin{equation}\label{Eq Conservative Equation}
	\div \v = 0  , \quad \text{in }   \Omega .
	\end{equation}
	\begin{equation}\label{Eq Drained Boundadry Conditions}
	p = 0 , \quad\text{on }  \partial \Omega . 
	\end{equation}
\end{subequations}
Here, $a
$ is the \textbf{flow resistance} (the fluid viscosity times the inverse of the medium permeability tensor), $P$ is the external pressure potential, $\g$ is the gravity, $\rho$ is the fluid density and $ \zeta $ is the function defining the fissure in Definition \ref{Hyp Fissure Geometric Conditions}. The Problem \eqref{Eq Porous Media Equation} above is composed by the Darcy constitutive equation \eqref{Eq Darcy Equation}, the mass conservation equation \eqref{Eq Conservative Equation} and the drained (Dirichlet) boundary conditions \eqref{Eq Drained Boundadry Conditions}. The external pressure gradient $\grad P$ and the geometric gradient $\grad \zeta$ are incorporated as driving forces in Equation \eqref{Eq Darcy Equation}. Finally, Equation \eqref{Eq Conservative Equation} states that the system is free of sources and therefore conservative. 

It is a well-known fact that the System \eqref{Eq Porous Media Equation} is well-posed (see \cite{Gatica, BraessFEM}) in the continuous case, consequently it can be solved in a discrete setting associated to a sampling triangulation $\triang$ using numerical analysis methods e.g., mixed finite element methods (see, \cite{Gatica}). In order to attain the numerical solution of Equation \eqref{Eq Porous Media Equation} consistent with the requirement $ \v \in \V(\Omega,\triang) $, we must define the spaces to which the numerical approximations of $ p,  P $ and $ \zeta $ must belong. The \textbf{linear space} associated to $\Omega, \triang$ is defined as
\begin{equation}\label{Eqn Linear Affine Space}
P_{1}(\Omega, \triang) \defining\Big\{
q = \sum_{K \, \in \,\triang} q_{K} \, \ind_{K} : q \quad \text{continuous and }\, 
q_{K} \in P_{1}(K), \quad \text{for all } \, K\in \triang \Big\} ,
\end{equation}
where
\begin{equation}\label{Eqn Linear Affine Polynomials Element}
P_{1}(K) \defining\big\{
ax + b y + c
 : a, b, c\in \R
\big\}.
\end{equation}
Clearly, the approximated surface $ \widetilde{\zeta} $ belongs to $ P_1(\Omega, \triang) $. Moreover, if $ q \in P_1(\Omega,\triang) $, then $ \grad q \in \H(\Omega, \triang) $. 

\subsection{Construction of the master conservative velocity field $ \v $}

Let $ P \in P_1(\Omega, \triang) $ be given. The construction of the master conservative flow field $ \v \in \V(\Omega,\triang) $ is accomplished in three steps. First, compute (see \cite{BradjiHerbin, BraessFEM, LeVeque}) the unique numerical solution $ p_0 \in P_1(\Omega, \triang)$ to 
\begin{align}\label{Eqn Direct Formulation Porous Media}
 -\div \big( \grad p\big) &=  \rho g \, \div \grad \widetilde{\zeta} + \div \grad P \quad \text{in } \Omega, \\
  p &= 0\quad \text{on } \partial \Omega .
\end{align}
Next, compute the \textbf{primary velocity field} $\v_{0} \in \H(\Omega, \triang)$ associated to $\Gamma$ as
\begin{equation}\label{Def Primary Flow Field}
\v_{0} \defining - \frac{1}{a} \sum\limits_{K\, \in \, \triang} 
\big( \grad p_{0} 
+ \grad P
+ \rho g \,\grad \widetilde{\zeta} \, \big) \, \ind_{K}.
\end{equation}
Finally, in order to guarantee conservation of mass, define the \textbf{master conservative velocity field} $\v\in \V(\Omega, \triang)$ as
\begin{equation}\label{Def Conservative Flow Field}
\v \defining 
\text{Orthogonal Projection of the primary velocity field }\, \v_{0} \in \H(\Omega, \triang)\; \text{onto the space }\, \V(\Omega, \triang).
\end{equation}

Theorem \ref{Th Projection Characterization} below provides a simple linear algebra approach for the computation of $ \v $. First we need some auxiliary definitions.

\begin{definition}\label{Def The Characterizing Matrix}
Let $ \{K_{\ell}: 1 \leq \ell \leq \#\triang\} $ and $\{ e_{i}: 1\leq i \leq \#\E\}$ be two enumerations of the sets $\triang$ and $\E$ respectively. 
\begin{enumerate}[(i)]
\item Given any element-wise constant vector field $ \w =  \sum\limits_{\ell \, = \,1}^{\#\triang} \w_{K_{\ell}} \, \ind_{K_{\ell} }  $, the following $ \Numb $ map gives the concatenation of its values as column vectors: 
\begin{equation}\label{Eq Numbering Map}
 \Numb: \H(\Omega, \triang) \rightarrow \R^{2\#\triang}, \quad
 (\Numb\w)\, (j) \defining \begin{cases}
\w_{K_{\ell}} \cdot \iversor \, , & j = 2\, \ell - 1 \, , \\
 \w_{K_{\ell}} \cdot \jversor \, , & j = 2\, \ell  \, .
\end{cases}
\end{equation}

\item The \textbf{characterizing matrix} $A\in \R^{\#\E \times 2\#\triang}$ of the triangulation $ \triang $ has the entries of the normal outwards vectors $ \nuhat_{K \vert L } $, organized as follows
\begin{align}\label{Eq The Characterizing Matrix}
A (i, j) \defining \begin{cases}
\nuhat_{K_{\ell}|L_i}\cdot \iversor \, ,  & j = 2\ell - 1  , \, K_{\ell} \cap L_i \neq \emptyset \, ,\\
\nuhat_{K_{\ell}|L_i}\cdot \jversor \, ,  & j = 2\ell , \, K_{\ell} \cap L_i \neq \emptyset \, ,\\
0 , & \text{otherwise} \, .
\end{cases}
\end{align}
\end{enumerate}
\end{definition}

\begin{theorem}\label{Th Projection Characterization}
Let $\v_{0}$, $\v$ be the primary and master conservative velocity fields as defined in Equations \eqref{Def Primary Flow Field}, \eqref{Def Conservative Flow Field} respectively, then
\begin{equation}\label{Eq Commutativity of Operators} 
\v= (\Numb)^{-1}\Big(I - A^{T}\big(A A^{T}\big)^{-1} A\Big) \Numb\v_{0} .
\end{equation}
\end{theorem}

\begin{proof}
Note first that $\Numb\big[\V(\Omega, \triang)\big] = \ker(A)$, and therefore $ \dim \big[\V(\Omega, \triang)\big] = \dim \big[\ker (A)\big] > 0  $. Indeed, if $\w \in \V(\Omega, \triang)$, and $K, L_i \in \triang$ are two triangles sharing the edge $e\in \E_{\itfc}$, then the constraint in Definition \eqref{Eqn Conservative Space} of $\V(\Omega, \triang)$ can be written as  $\mathbf{a}_{i}\cdot \Numb (\w) = 0$ where $ \mathbf{a}_i $ is the is the $i$-th row vector of the matrix $A$. Since this holds for all $e\in \E_{\itfc}$, this is equivalent to $ A \, \Numb (\w) = \mathbf{0} $. The fact that $ \ker(A) $ is not trivial, follows from the inequality $\# \E_{\itfc} < \dfrac{3}{2} \, \#\triang < 2 \, \#\triang$.

In order to obtain \eqref{Eq Commutativity of Operators} from \eqref{Def Conservative Flow Field} it is enough to find a matrix representing the linear transformation $\Proj_{\ker (A)}$. 
From standard linear algebra theory, the fundamental spaces of the matrix $A$ satisfy
\begin{align}\label{Eq Orthogonality Fundamental Spaces}
& \ker(A) \perp \col\big(A^{T}\big), & 
& \ker(A) \oplus  \col\big(A^{T}\big) =  \R^{2\#\triang} ,
\end{align} 
where $\col\big(A^{T}\big)$ is the column space of $A^{T}$ and $\oplus$ indicates the direct sum of vector spaces. Therefore, the projection is characterized by
\begin{equation*}
 \big\Vert \x - \Proj_{\col (A)} \x  \big\Vert 
= \min \big\{ \big\Vert \x - A^{T}\y \big\Vert :\y\in \R^{\#\E} \big\}, \quad
 \text{for all }\, \x \in \R^{2\#\triang}\, ,
\end{equation*} 
with $\Vert \cdot \Vert$ the second order mean norm. From least squares standard theory we know that $\Proj_{\col (A)} \x = A^{T}\,\big(AA^{T}\big)^{-1}A \, \x$ for all $\x\in \R^{2\#\triang}$. Finally, due to the orthogonality of the fundamental spaces in \eqref{Eq Orthogonality Fundamental Spaces}, we have that
\begin{equation*}
 \Proj_{\ker (A)} \x = \Big(I - A^{T}\,\big( AA^{T}\big)^{-1}A\Big) \x, \quad   
 \text{for all }\, \x \in \R^{2\#\triang}\, . 
\end{equation*} 
From here the Identity \eqref{Eq Commutativity of Operators} follows trivially.
\end{proof}

\begin{remark}\label{Rem Necessity of Projections}
A different approach would compute a numerical solution $(\v_{0}, p_{0})$ to Problem \eqref{Eq Porous Media Equation} using a dual mixed finite element method (see \cite{BraessFEM, Gatica}). In this case $\v_{0}$ already belongs to a Raviart-Thomas space $\boldsymbol{RT}_{\!\ell}(\Omega, \triang)$ of certain degree $\ell\geq 0$. In particular, the primary velocity field $\v_{0}$ is already numerically conservative; consequently, it is equal to the master conservative velocity field i.e., $\v = \v_{0}$ and, there is no need of the projection used in Theorem \ref{Th Projection Characterization}. 
\end{remark}

\subsection{The lifting operator and the global velocity field $ \u $}

Our next step is to define, based on the master flow field $\v$, a new conservative three-dimensional flow field $ \u  \in \widetilde{V}(\Omega, \triang)$. The basic procedure consists in, for each element $K\in \triang$, find an appropriate 3-D axis around which rigidly rotate horizontal vectors, so that $ \v_K $ ends up being parallel to the lifted triangular element $ K^{\zeta} $. In addition, the magnitude of the velocity vector, normal to the boundary of the element, is not altered and mass conservation is thus ensured. 

Fix an element $ K \in \triang $ and let $ \n_K $ be the unit normal vector to the corresponding element $ K^{\zeta} \in \Gammatilde $ as in Definition \ref{Def lifting}, see Figure \ref{Fig vk2uk}. Let $ \v_K $ denote the value of the master flow field $ \v $ on $ K $. If $ \Kz $ is horizontal, we simply set $ \u_K = \v_K $. Otherwise, we choose as a rotation axis the following vector
\begin{equation}
  \zhat_K\defining  \frac{\n_K\times \kversor}{\vert\n_K\times \kversor\vert}. 
 \end{equation}
The component of $ \u_K $ that is parallel to $ \zhat_K $ is set equal to that of $ \v_K $, namely $ \Proj_{\zhat_K}(\u_K) = \Proj_{\zhat_K}(\v_{K})$. Now, let $ \Lambda_K $ be the plane normal to $ \zhat_K $, and let $ \theta_K $ be the angle formed from $ \kversor $ to $ \n_K $ on the plane $ \Lambda_K $. The component of $ \u_K $ perpendicular to $ \zhat_K $ is then taken equal to the vector obtained from rotating $ \Proj_{\Lambda_k}(\v_{K})$ by the angle $ \theta_K $.

\begin{figure}\centering
\includegraphics[scale = 0.85]{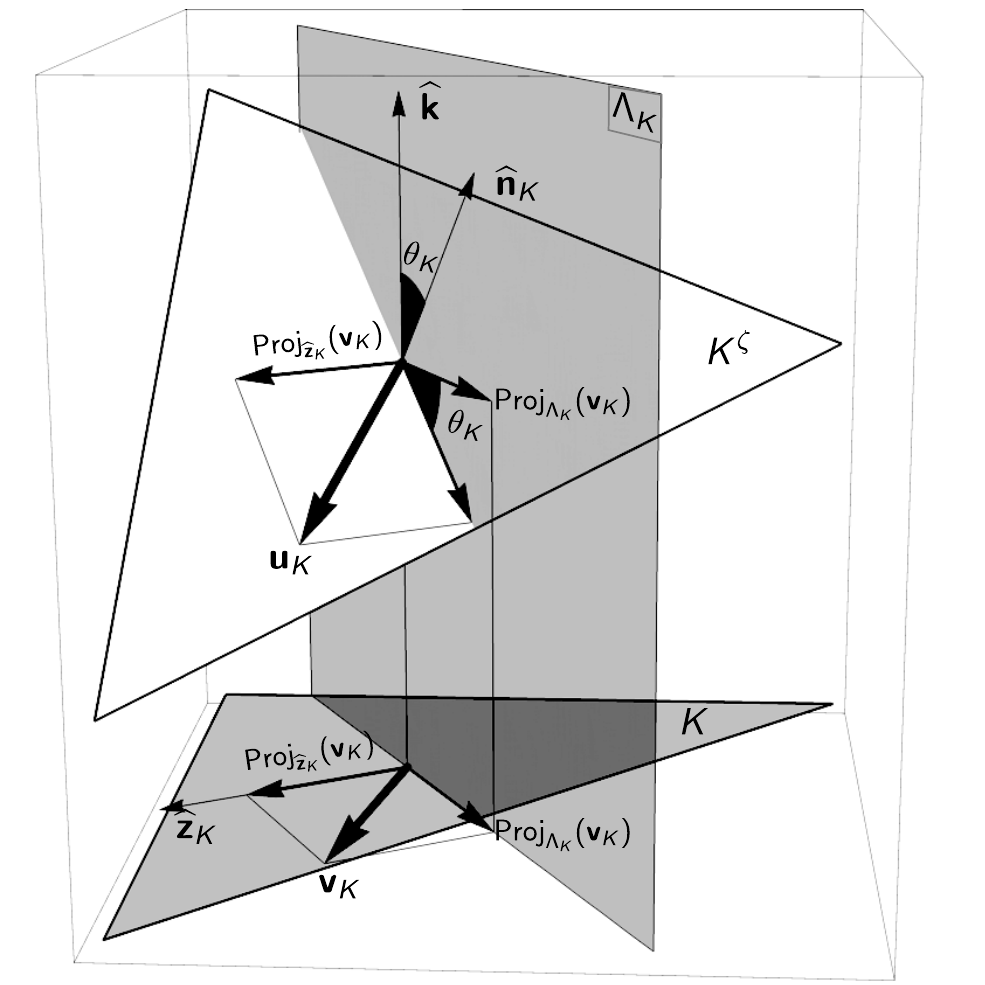}
\caption{Schematics of the lifting of the conservative field $ \v_K $ to the global velocity field $ \u_K $ for a fixed element $ K $. \label{Fig vk2uk}}
\end{figure}

The \textbf{lifting operator map}, $F_{K}: \R^{2}\rightarrow \R^{3}$, $ \v_{K}\mapsto \u_K$ is linear and it is explicitly written as:
\begin{equation}\label{Eq local velocity due curvature 3-D}
 \u_K 
\defining \left[\begin{array}{cc}
1- \dfrac{n_{1}^{\,2}}{1+n_{\,3}}   & -\dfrac{n_{1}\,n_{\,2}}{1+n_{\,3}}\\[8pt]
-\dfrac{n_{1}\,n_{\,2}}{1+n_{\,3}} & 1- \dfrac{n_{\,2}^{\,2}}{1+n_{\,3}}\\[8pt]
-n_{1} & -n_{\,2}
\end{array}
\right]
\v_K \defining F_{K} \v_K
, \quad \n_K =\left[\begin{array}{c}
									n_1\\[1ex]
									n_2\\[1ex]
									n_3
					\end{array}\right].
\end{equation} 
The \textbf{global velocity field} $\u$ is then defined by 
\begin{equation}\label{Eq global velocity due curvature 3-D}
\u 
(\x) \defining \sum_{K \,\in \, \triang} \u_K
\ind_{K}(\x), \quad \x \in \Omega,
\end{equation}

\begin{remark}\label{Rem comments on the velocity field}
\begin{enumerate}[(i)]
\item The construction of the lifting operator $F \defining \sum\limits_{K \,\in \, \triang} F_{K}
\ind_{K}(\x), \quad \x \in \Omega$ \eqref{Eq local velocity due curvature 3-D} is based on preserving the magnitudes of the vectors on every direction normal to the boundary of $ \Kz $. Consequently, since the master conservative field $\v$ belongs to $\V(\Omega, \triang)$ then, its lifting $\u$ belongs to $ \widetilde{\V}(\Omega, \triang) $, 
i.e. $\u$ is conservative.

\item 
Given two master conservative fields $\v_{1}, \v_{2} \in \V(\Omega, \triang)$, a direct calculation shows that 
\begin{equation}\label{Eq orthogonality conservation}
\v_{1} (\x) \cdot \v_{2}(\x) =
\u (
\v_{1}, \x)\cdot \u(
\v_{2}, \x)
 \quad \text{for all}\; 
 \x\in \Omega.
\end{equation}
\end{enumerate}
\end{remark}

\subsection{Mean streamline length}\label{Sec Mean streamline}
Within each lifted element $ K^\zeta $, the streamlines of the field $ \u $ are parallel. In the context of energy dissipation discussed in Section \ref{Sec Applications}, we will need the average length $ \boldsymbol{d}_K(\u) $, of such streamlines. 

Consider first the triangle $ K \in \triang $ shown in Figure \ref{Fig calculation_dk}. It is an elementary fact that the average streamline length within $ K $ of the field $ \v_K $ is equal to $ d_K(\v) = \tfrac{1}{2} \Vert \boldsymbol{y}_K(\v) \Vert $ where $ \Vert \boldsymbol{y}_K(\v) \Vert$ is the maximum length of segments parallel to $ \v_K $ contained in $ K $. Clearly $ y_K(\v) = \alpha_K \v_K $ for some $ \alpha_K > 0 $. The desired distance is then obtained by applying the operator $ F_{K} $ of \eqref{Eq local velocity due curvature 3-D} to $ \boldsymbol{y}_K $,
\begin{equation}\label{Def mean streamline distance}
\boldsymbol{d}_K(\u) =  \frac{1}{2} \, \Vert F_{K}  \boldsymbol{y}_K(\v) \Vert = \frac{1}{2} \, \alpha_K \Vert \u_k \Vert.
\end{equation}

We now provide an algorithm to compute $ \alpha_K $. Let $ e_1,e_2,e_3 $ denote the edges of element $ K $, and consider the products $ f_i = (\nuhat_{e_i} \cdot \v_K) $. If any of the $ f_i $ equals zero, then there is an edge vector $ \boldsymbol{e}_i $ parallel to $ \v_K $ and we make $ \boldsymbol{y}_K(\v) = \boldsymbol{e}_i $, $ \alpha_K = \Vert \boldsymbol{e}_i \Vert / \Vert \boldsymbol{v}_K \Vert $. Otherwise the $ f_i $ satisfy the following property: there is a unique edge $ e_{\v,1} $ for which its corresponding $ f_{\v,1}=(\nuhat_{e_{\v,1}} \cdot \v_K) $ has sign different from the other two, namely $ f_{\v,1} $ is the unique strictly negative or strictly positive number among $ f_1, f_2, f_3 $. The vector $ \v_K $ enters or exits $ K $ through the edge $ e_{\v,1} $ according to $ \text{sign}(f_{\v,1}) $. Let $ \boldsymbol{e}_{\v,1}, \boldsymbol{e}_{\v,2}, \boldsymbol{e}_{\v,3} $ denote the edge vectors of $K$ transversing $ \partial K $ in the counter-clockwise direction and starting with the selected edge $ e_{\v,1} $. There exist unique positive constants $ \alpha_K, \beta_K $ such that
\begin{equation}
\boldsymbol{e}_{\v,3} + \beta_K \boldsymbol{e}_{\v,1} - \text{sign}(f_{\v,1}) \alpha_K \v_K = \boldsymbol{0}.
\end{equation}

\begin{figure}\centering
\includegraphics[scale = 0.9]{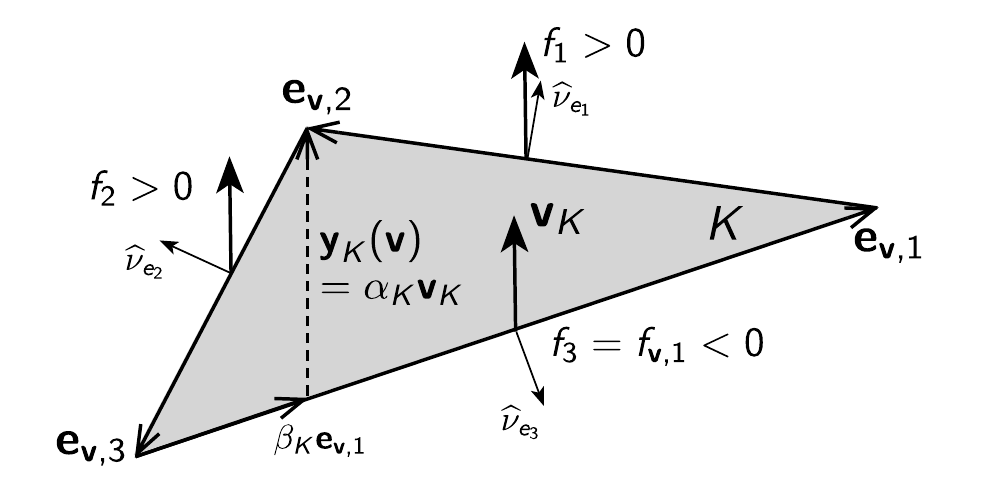}
\caption{Calculation of the mean streamline length of the field $ \v_K $. The vector $ \boldsymbol{y}_K(\v) = \alpha_K \v_K $ is shown as dashed. In this case the $ \boldsymbol{y}_K $ enters $ K $ through $ e_{\v,1} $ because $ f_3 = f_{\v,1} < 0 $. \label{Fig calculation_dk}}
\end{figure}

%
\section{Transport along the surface}\label{Sec Markov Chains}
%
%
In this section we derive a stochastic model for the linearized transport of a solute which is being advected along the surface by the global velocity field $ \u $. The starting point is the linear transport equation for a solute of concentration per unit volume $ c(x,t) $ which is being advected by the velocity field $ \u $ and initially distributed according to an initial concentration $ c_0 $:
\begin{align} \label{Eqn Transport Equation}
& \partial_{t} c(\x, t) + \div (\u \, c)(\x,t) = 0 \quad  c(\x,0) = c_{0}(\x), \quad \x \in \widetilde \Gamma_{ \Omega, \triang}  .
\end{align}

Upon discretization of the surface by the triangulation $ \triang $ and integration over each element, the following approximation to \eqref{Eqn Transport Equation} is obtained
\begin{equation}\label{Eqn Discrete Transport}
\frac{d \mathbf{c}}{dt}  = Q \mathbf{c}(t), \quad \mathbf{c}(0) = \mathbf{c}^{(0)},
\end{equation}
where now $ \mathbf{c} $ and $ \mathbf{c}^{0} $ denote vectors in $ \mathbb{R}^{\# \triang_{\ext}} $ with entries given by the total mass of solute over each element, 
\begin{equation}
c_K(t) = \int_{K^\zeta \times D} c(\x,t) d \x, \quad c^{(0)}_K = c_K(0).
\end{equation}
The integration is carried over the volume $ K^\zeta \times D $ where $ 0< D \ll 1 $ is the depth of flow (i.e. the thickness of the fissure) and is assumed to be constant throughout. The matrix $ Q \in \R^{\# \triang_{\ext} \times \#\triang_{\ext}}$ in \eqref{Eqn Discrete Transport} corresponds to the discretized version of the divergence operator, and is given by

\begin{align}\label{Def Q}
Q_{K,K} &  \defining 
- \!\!\! \sum_{\substack{L \in \triang_{ext} \\ |\partial L \cap \partial K| >0 }} 
  \frac{\sigma^\zeta_{K|L}}{|K^\zeta|}  \; (\u_K \cdot \nutilde_{K|L})^+, \quad K \in \triang_{\ext};\\
Q_{K,L} &\defining 
  \frac{\sigma^\zeta_{K|L}}{|K^\zeta|} \; (\u_K \cdot \nutilde_{K|L})^+ , \quad  \text{ if } K,L \in \triang_{\ext}, \; K \neq L, \;  |\partial K \cap \partial L| > 0, 
\end{align}
where $ z^+ = \max\{ z, 0\} $ for $ z \in \R. $

The solution to \eqref{Eqn Discrete Transport} is
\begin{equation}\label{Eqn Solution Discrete Transport}
\mathbf{c}(t) = \exp(Q t) \, \mathbf{c}^{(0)} \defining T(t) \mathbf{c}^{(0)}, \quad t \geq 0.
\end{equation}

\begin{remark}
\begin{enumerate}[(i)]
\item Each entry of the matrix $ Q $ has units of fraction of solute mass per unit time. The rows corresponding to elements in $ \triang_{\ext} - \triang $ contain zero in all its entries.
\item Note that because of the construction of the global velocity field, the terms $ (\u_K \cdot \nutilde_{K|L}) $ can be replaced throughout \eqref{Def Q} by the dot product $ (\v_K \cdot \nuhat_{K|L}) $ of vectors in $ \mathbb{R}^2 $. 
\item  In contrast, the cross sectional area used in the computation of $ Q $ must be computed with the distance $ \sigma_e^\zeta $. 
\end{enumerate}
\end{remark}

It is important to point out that, in the case where $ \sum\limits_{K \, \in \, \triang} c^{(0)}_K  = 1$, the concentration $ \mathbf{c}(t) $ in \eqref{Eqn Solution Discrete Transport} can be understood as the distribution of a stochastic process $ X = \{X(t): t\geq 0\} $ initially distributed according to $ \mathbf{c}^{(0)} $ and with phase space given by the elements of $ \triang_{\ext} $. Namely,
\begin{equation}\label{key}
c_K(t) = \prob \left(X(t) = K \big| X(0) \sim \mathbf{c}^{(0)} \right)
\end{equation}
and $ X(t) $ thus represents a random model for the position of an individual molecule of the solute being advected. The elements in $ \triang_{\ext} - \triang $  are absorbing states. The randomness in this stochastic process comes from the uncertainty in the geometry of the fissure which is reflected in the approximation of $ \u_K $ to the real velocity field on the points of the surface. In this context, the matrix $ Q $ is the ``infinitesimal generator'' of $ X $ while the family of ``transition probability matrices'' of the process $\{ T(t): t \geq 0\} $, defined in Identity \ref{Eqn Solution Discrete Transport}, is given by
\begin{equation}\label{Eqn trasition probs}
	T_{K,L}(t) \defining \prob\big(X(t) = L \big| X(0) = K\big), \quad K,L \in \triang_{\ext}.
\end{equation}
 See Chapter IV in \cite{WaymireBhattacharya}.

\subsection{The geometry of transport}

The matrix $ Q $ provides a description of the rates of transport forced by the velocity field $ \u $ along the surface. It therefore gives detailed information about which parts of the surface are accessible to the flow from any initial position. To be precise, let $ \widetilde{Q} $ be the transition probability matrix of the Markov chain associated to $ X $ (see \cite{WaymireBhattacharya}, V.5)
\begin{align}\label{Def Qtilde}
\widetilde{Q}_{K,L} = \frac{Q_{K,L}}{-Q_{K,K}} \; \text{ if } Q_{K,K} \neq 0, &\quad 
\widetilde{Q}_{K,L} = 0 \text{ if } Q_{K,K} = 0, \quad  K\neq L;\\
\widetilde{Q}_{K,K} = 0 \; \text{ if } Q_{K,K} \neq 0, &\quad  
\widetilde{Q}_{K,K} = 1 \text{ if } Q_{K,K} =0
\end{align}
and let $ \graph $ be the directed weighted graph whose adjacency matrix is $ \widetilde{Q} $. We will refer to $ \graph $ as the ``flow graph'' associated with the velocity field $ \u $. 

We say that the element $ K $ is upstream from the element $ L $ and denote it by $ K \to L $, if $ (\widetilde{Q}^n)_{K,L} > 0 $ for some $ n \geq 1 $, or equivalently, there exists a path in $ \graph $ connecting $ K $ with $ L $. It follows from the standard theory of Markov processes that $ K \to L $ implies $ T_{K,L}(t) > 0 $ for all $ t > 0 $. In other words, it is likely for a particle originally within triangle $ K $ to be transported to the element $ L $.

A particular feature of the flow graph $ \graph $ generated by $ \u $ is that it is a forest, a collection of trees similar to a river network driven by the gravitational potential over the landscape.

\begin{theorem}\label{Th cycle free graph}
The unique cycles in $\graph$ have length one. In particular, no element $K\in  \triang$ is downstream of itself.
\begin{proof}
First we observe that for any $\v\in \V(\Omega, \triang)$ there exists $q\in E(\Omega, \triang)$ such that $\v = - \grad q$ (see Identity \eqref{Eqn Conservative Potential Space}). Now we proceed to prove the result by way of contradiction, let $\{L_{i}: 0 \leq i \leq j \}$ be a cycle in $\graph$. Define $K \defining L_{0} = L_{j}$ and denote by $e_i = L_{i-1}|L_{i}$, $1\leq i \leq j$ the sequence of edges joining subsequent elements of the cycle . Also define the midpoints $\x_{i}$ of $e_{\,i}$ and the unitary vectors $\nuhat_{e_i}$ outwards $L_{i - 1}$ and orthogonal to $e_{\,i}$. Notice that since $L_{0}, L_{1}, \ldots, L_{j}$ are successors then $Q_{L_{i - 1}, L_{i} } > 0$ for $i = 1, 2, \ldots, j$. In particular, since $\v_{L_{i}} = - \grad q\big\vert_{M_{i}} $, due to the definition of $Q_{M_{i - 1}, M_{i} }$ this implies that $q(x_{i - 1}) \lneqq q(x_{i})$ for all $i = 1, 2, \ldots, j$. Therefore, because of the continuity of $q$ it follows that 
\begin{equation*}
q(x_{0}) \lneqq q(x_{1}) \lneqq \ldots \lneqq
q(x_{j}) \lneqq q(x_{0}) .
\end{equation*}
Since this is a contradiction, the proof is complete.
\end{proof}
\end{theorem}

\subsection{The time to reach the boundary}

One advantage of the stochastic formulation is that we can explicitly consider dynamic quantities regarding the motion of individual solute particles, and their aggregated properties. Most significantly, consider the random time it takes a particle to reach $ \partial \triang$,
\begin{equation}\label{Def exit time}
\tau = \inf\{t \geq 0: X(t) \in \partial \triang \}
\end{equation}
and its distribution conditioned to having started within element $ K \in \triang $, which can be characterized via
\begin{equation}\label{Def Distribution tau}
	\Phi_K(t) \defining\prob\big(\tau>t \big| X(0) = K \big), \quad K \in \triang.
\end{equation}

The states in $ \triang_{\ext} - \triang $ are absorbing, then it follows that
\begin{equation}\label{Eqn Distribution tau}
	\Phi_K(t) = \sum_{L \, \in \, \triang} T_{K,L}(t), \quad \Phi(t) = T(t) \ones_\triang
\end{equation}
where, in the second equation, $ \Phi(t) $ denotes the vector with entries $ \{\Phi_K(t): K \in \triang\} $ and $ \ones_\triang $ is the vector with ones on the entries corresponding to the triangles in $ \triang $ and zeros everywhere else, see \cite{WaymireBhattacharya} IV.10.

Since $ \tau $ is a positive random variable, its expectation can be computed directly from $ \Phi $. Define 
\begin{equation}\label{Def Expectation tau}
	\Psi_K \defining \mathbb{E}(\tau | X(0) = K), 
\end{equation}
and let $ \Psi $ be the vector with entries $ \{\Psi_K(t): K \in \triang\} $. One thus arrives to the following useful formula:
\begin{equation}\label{Eqn Psi}
	\Psi = \int_0^\infty \exp(Q t) \, \ones_\triang \, dt.
\end{equation} 
Even though it can be shown that $ \Psi < \infty $, no closed formula is possible for the integral in \eqref{Eqn Psi} because $ Q $ is by construction, a singular matrix.

%
%
%
%
\section{Applications and Numerical Examples}\label{Sec Applications}
%
%
In this section we present three important applications of constructed the flow field. The first one is the computation of the preferential flow direction for a layered medium. The second one is a computation of energy dissipation hosted on the surface. The third one is a computation of probable stream lines through the surface. 
%
%
%
%
\subsection{Preferential Flow Direction}\label{Sec Preferential Flow Direction}
%

%
Let $ \u $ be the global velocity field defined in \eqref{Eq global velocity due curvature 3-D}. The \textbf{preferential flow direction} of the medium is given by
\begin{align}\label{Eqn Preferential Flow Direction}
& \m[\u] \defining \frac{1}{\vert \widetilde{\Gamma} \vert }
\sum_{K \, \in\, \triang } | \Kz | \, \u_K.  
\end{align}
The vector $ \m[\u] \in \R^3 $ is an estimate of the mean velocity of whole fissure.
%
%
\subsection{Energy Dissipation}\label{Sec Energy Dissipation}
%
%
Three separate physical mechanisms of energy dissipation are considered: Curvature, Friction and Gravity. For each mechanism, a functional of energy dissipation is constructed for computational purposes. 
\begin{definition}\label{Def Curvature Energy Dissipation}
 The \textbf{energy dissipation rate due to curvature} is given by the total kinetic energy of such deviation, weighted according the flow rate (see \cite{Batchelor, Bear}), i.e.

\begin{equation}\label{Def curvature energy dissipation}
U_{\curv} (\u) \defining 
 \frac{1}{ 2 } \, \rho D \, \displaystyle{\sum_{K \in \triang}} \; 
\sum_{\substack{L \in \triang \\ |\partial L \cap \partial K| >0 }}
 \big(|K^\zeta| \, \widetilde{Q}_{K,L} + |L^\zeta| \, \widetilde{Q}_{L,K} \big)\, 
\big\Vert ( \u_K\cdot \nutilde_{K|L}) \; \nutilde_{K|L}
- (\u_L \cdot \nutilde_{L|K}) \; \nutilde_{L|K} \big\Vert^{2}. 
\end{equation}

Here, $\rho$ is the density of the fluid, $g$ is the gravity.
\end{definition}
\begin{remark}
\begin{enumerate}[(i)]
\item
Notice that although the conservative flow field $\v$  is continuous in the normal direction across the interfaces $e\in \E_{\itfc}$. Due to the lifting, the vertical component of 
the global flow field $\u$ deviates from one element to its neighbor, because of the curvature of the surface $\Gamma$. Therefore, certain amount of  energy has to invested for the fluid to reach the configuration of its neighbor, the Identity \eqref{Def curvature energy dissipation} quantifies the total of this needed energy per unit time.

\item
Observe that the factor $D \big(|K^\zeta| \, \widetilde{Q}_{K,L} + |L^\zeta| \, \widetilde{Q}_{L,K} \big)$ in the Identity \eqref{Def curvature energy dissipation} quantifies the fraction of volume per unit time that crosses the edge $K|L$ because one of the summands is necessarily null, i.e. $\widetilde{Q}_{K,L}  \widetilde{Q}_{L,K} = 0$.

\end{enumerate}
\end{remark}
Next we define the friction and gravity dissipation energy functionals. Recall that the dissipation head of mechanical energy due to friction (see \cite{ZhangNemcik}) is given by the Darcy-Weisbach type equation 
\begin{equation}\label{Eq Darcy-Weisbach}
\gamma \, \frac{ l }{D} \, \frac{1}{2g}\,\Vert \u \Vert^{2} .
\end{equation}
Here $\gamma$ is a dimensionless friction coefficient depending on the material of the walls of the fissure and the viscosity of the fluid, $ l $ is the average length of the fluid path, and $D$ is the layer height and $\Vert \u\Vert$ is the average magnitude of the velocity. In order to convert the head \eqref{Eq Darcy-Weisbach} into a dissipation rate we will need the total discharge flowing out of element $ K^\zeta $ defined as:%

\begin{equation}\label{Eq qK}
	\boldsymbol{q}_K \defining - D |K^\zeta| \, Q_{K,K}
\end{equation}
where $ Q_{K,K} $ is defined in \eqref{Def Q}.

\begin{definition}\label{Def Friction Energy Dissipation}
The total \textbf{energy dissipation rate due to friction} is given by
\begin{equation}\label{Def friction energy dissipation}
U_{\fric} (\u) \defining 
\frac{1}{2} \, \rho\, \gamma \sum_{K\,\in\, \triang} 
\boldsymbol{q}_K\,
 \frac{\boldsymbol{d}_{K}(\u)}{D} \, \Vert \u_{K} \Vert^{2} .
\end{equation}
Here $\boldsymbol{q}_K$ is the discharge rate outwards the element $\Kz$ defined in the Identity \ref{Def Q} and $\boldsymbol{d}_{K}(\u)$ is the mean streamline length constructed in section \ref{Sec Mean streamline}. 
\end{definition}
For the quantification of the energy dissipation due to gravity, we fix $ K \in \triang $ and compute the increase or decrease of potential energy experienced by the flow field inside $ K $. The mean change in potential energy due to a layer of flow moving in the direction $ \u_{K} $ can be computed as $ \tfrac{1}{2} (F \boldsymbol{y}_K) \cdot \kversor = \tfrac{1}{2} \alpha_K (\u_K \cdot \kversor) $ where $ \tfrac{1}{2} \vec{y}_K(\v) = \tfrac{1}{2} \alpha_K \v_K$ is the mean streamline length found in section \ref{Sec Mean streamline}. We thus define:
\begin{definition}\label{Def Gravity Energy Dissipation}
The total \textbf{variation of potential energy rate} is given by 
\begin{equation}\label{Eq total external energy}
U_{\grav} (\u) 
\defining    \frac{1}{2} \,  \rho\, g   \sum_{K\,\in\, \triang}  \boldsymbol{q}_K\, \alpha_K (\u_K \cdot \kversor) 
\end{equation}
\end{definition}
\begin{remark}\label{Rem Nature of U grav}
Notice that unlike the physical mechanisms of curvature and friction where $U_{\curv}$, $U_{\fric}$ are quantities of energy dissipation, the functional $U_{\grav}$ may report and increment of potential energy depending on $\v$.  
\end{remark}
%
%
%
%
%
%
\subsection{Numerical Experiments}\label{Sec Numerical Experiments}
%
%
Here we present two numerical examples to illustrate the measurement of the physical quantities proposed throughout this section as well as the boundary reaching times. The two examples below correspond to different surfaces, in each case we perform two experiments, one with a null pressure potential, i.e., $ P = 0$ and the other with a logarithmic pressure potential $P = 4000 \, \log ((x - h)^2 + (y - k)^2)$, corresponding to a well centered at the point $C = \begin{pmatrix}
h \\ k
\end{pmatrix}$; notice that the pressure potential is the fundamental solution of the 2-D Laplace equation at the point $C$. 
For the graphics, we display side by side the potentials $\rho\, \g \, \zeta$, $\rho\,\g\, \zeta + P$, below each other the associated conservative flow fields are depicted and finally, the corresponding expected boundary reaching times per element are shown. The experiments are executed in a MATLAB script and use a Delaunay triangulation, composed of 322 elements for the first example and 1046 elements for the second example, where the vertices were generated randomly. 

Due to the scales of the phenomenon we choose the CGS system of units, for both examples we choose the values: water density $\rho = 100 \, kg/m^{3}$, gravity $g = 9.81 \, m/s^{2}$, water layer thickness $D= 0.01 m$, Darcy-Weishbach coefficient $\gamma = 0.03$ and flow resistance $a = 1.3071*10^{3} m^{3}s/g$, the geometric domain is the unit square i.e., $G = [0,1]^{2}$. In contrast, the numerical values describing the overall behavior of the surface, displayed in Tables \ref{Tb Example 1},  \ref{Tb Example 2} of Examples \ref{Ex Verification Example} and \ref{Ex Exploration Example} respectively, are presented in the CGS system, due to its scale. These values are the average velocity $\m[\u(\v)]$, the energy losses $U_{\curv}$, $U_{\fric}$, $U_{\grav}$ and the average expected time (average of the elements' expected time), denoted by $\m\big[\mathbb{E}(\tau | X(0) = K)\big]$. We also stress that in this context, positive amounts of energy indicate dissipation, while a negative sign means that such quantity is still available in the form of kinetic energy. 
\begin{example}\label{Ex Verification Example}
The first example is for verification purposes. The surface is given by the plane $\zeta: G\rightarrow \R$, $\zeta(x, y) =0.8\, (\sin(2\pi\, x) \,\exp(-2\pi\,y) + y)$. For the first experiment we set $P \equiv 0$ and for the second, a pressure potential $P = 4000  \log ((x - 110)^2 + (y + 10)^2 )$ i.e., an extraction well beyond the upper left corner of the graphic. Notice that both potentials are harmonic functions (see Figures \ref{Fig Surface and Time Expected Ex 1} (a) and (b) respectively), therefore the master velocity fields are a multiple the corresponding gradients and the global conservative fields are the lifting of such gradients multiplied by $\frac{1}{a}$. These facts can be observed on Figures \ref{Fig Surface and Time Expected Ex 1} (c) and (d). Also notice that the expected boundary reaching times are consistent with the a-priori intuitive notion as Figures \ref{Fig Surface and Time Expected Ex 1} (e) and (f) show. The next table summarizes the overall behavior of the surface for both potentials.
\begin{table}[h!]
\def\arraystretch{1.4}
\begin{center}
\begin{tabular}{ c c c c c c c}
    \hline
Experiment  & 
$\m[\u(\v)]$ &
$U_{\curv}$ & 
$U_{\fric} $ & 
$U_{\grav}$  &
Energy Balance &
$\m\big[\mathbb{E}(\tau | X(0) = K)\big]$ \\
           & 
[cm/s]  &
[erg/s] & 
[erg/s] & 
[erg/s] &
[erg/s] &
[s] \\[2pt]
\toprule
\shortstack{Gravity\\driven only} & 
$\begin{pmatrix}
   0.0028 \\
  -0.3548 \\
  -1.0626 \\
\end{pmatrix}$  &
   231.7744 &
   2123.0304 &
  -103578.4 &
  -101223.7 &
134 
\\[20pt] 
\shortstack{Gravity and \\pressure driven} & 
$\begin{pmatrix}
  -0.4228 \\
   0.0236 \\
  -0.83512 
\end{pmatrix}$  &
   262.4044 &
   2142.613 &
  -70541.36 &
  -68136.34 &
105 
\\ 
    \hline
\end{tabular}
\end{center}
\caption{Example \ref{Ex Verification Example}: Energy Dissipation Table, 322 Elements.}
\label{Tb Example 1}
\end{table}

\noindent From the table above it follows that for both potentials, the main driving force is the gravity. For the first case, the order of scales states that $U_{\curv} \ll U_{\fric} \ll U_{\grav}$. In the second case the well is meant to force the flow uphill, which is done up to an extent, this is reflected as the value of $U_{\grav}$ contracts in 32 percent approximately, while $U_{\curv}$ and $U_{\grav}$ do not change significantly with respect to itself hence, the order of scales is changed to $U_{\curv} \ll U_{\fric} \sim U_{\grav}$. Finally, it is observed that the average expected reaching time decreases in 20 percent from the first to the second case as now there are two main evacuation points.
%
\begin{figure}
        \centering
        \begin{subfigure}[Forcing Potential Surface $ \rho\, \g \, \zeta$.]
                {\resizebox{6.5cm}{6.5cm}
{\includegraphics{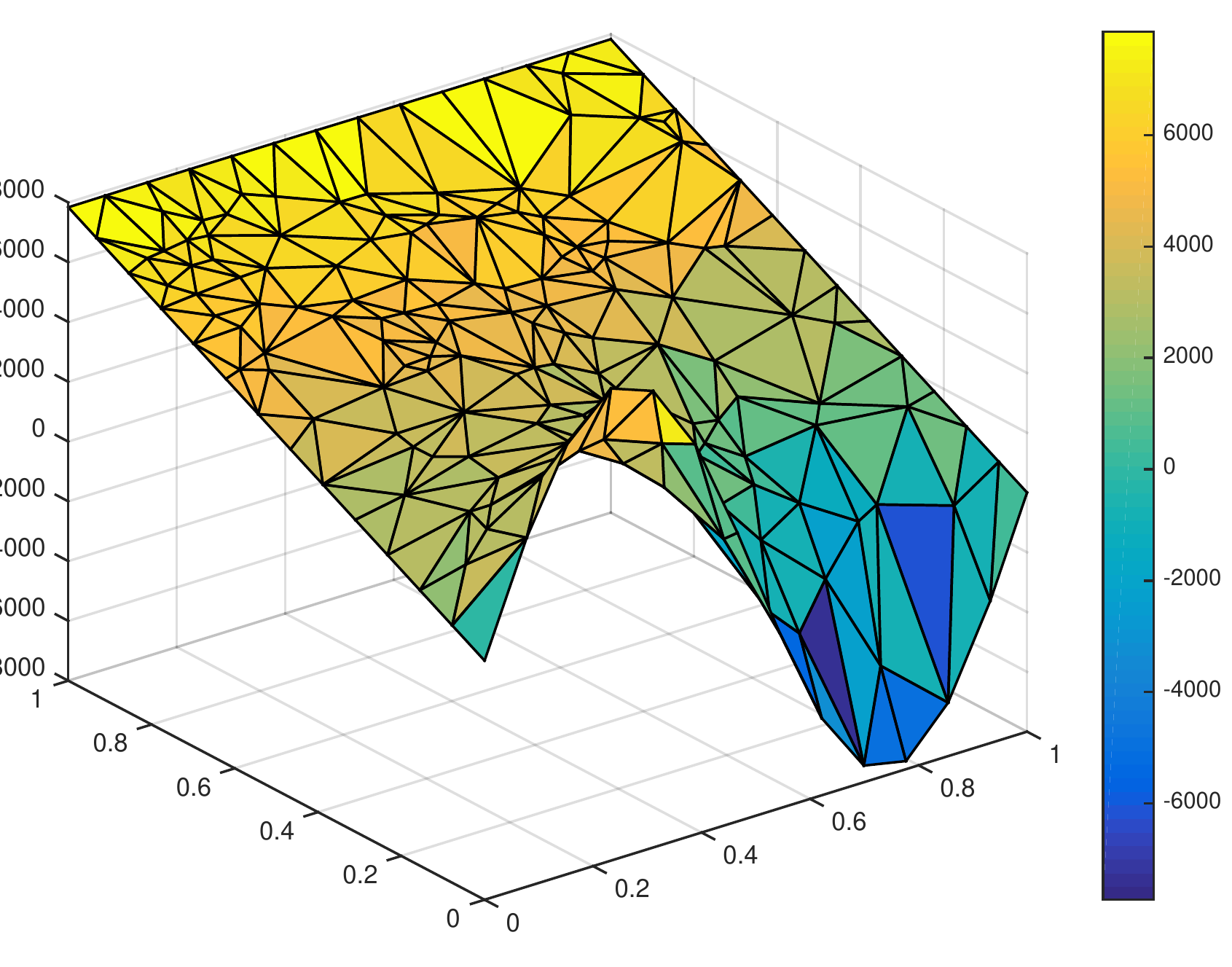} } }
        \end{subfigure}
        \qquad
        ~ 
          \begin{subfigure}[Forcing Potential Surface $\rho\,\g\, \zeta + P$.]
                {\resizebox{6.5cm}{6.5cm}
{\includegraphics{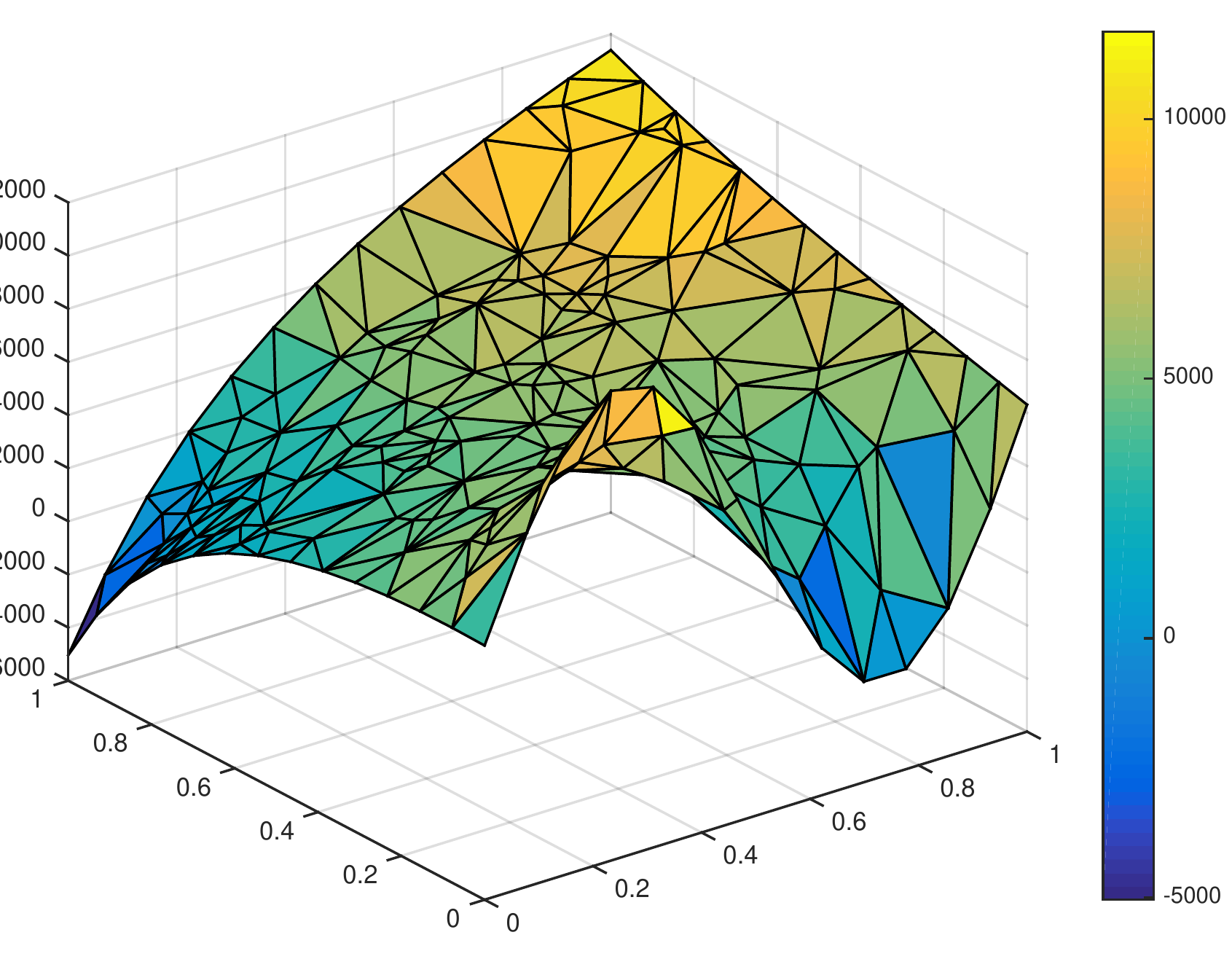} } }
        \end{subfigure}%
        \\
        \begin{subfigure}[Flow Field gravity driven flow.]
                {\resizebox{6.5cm}{6.5cm}
{\includegraphics{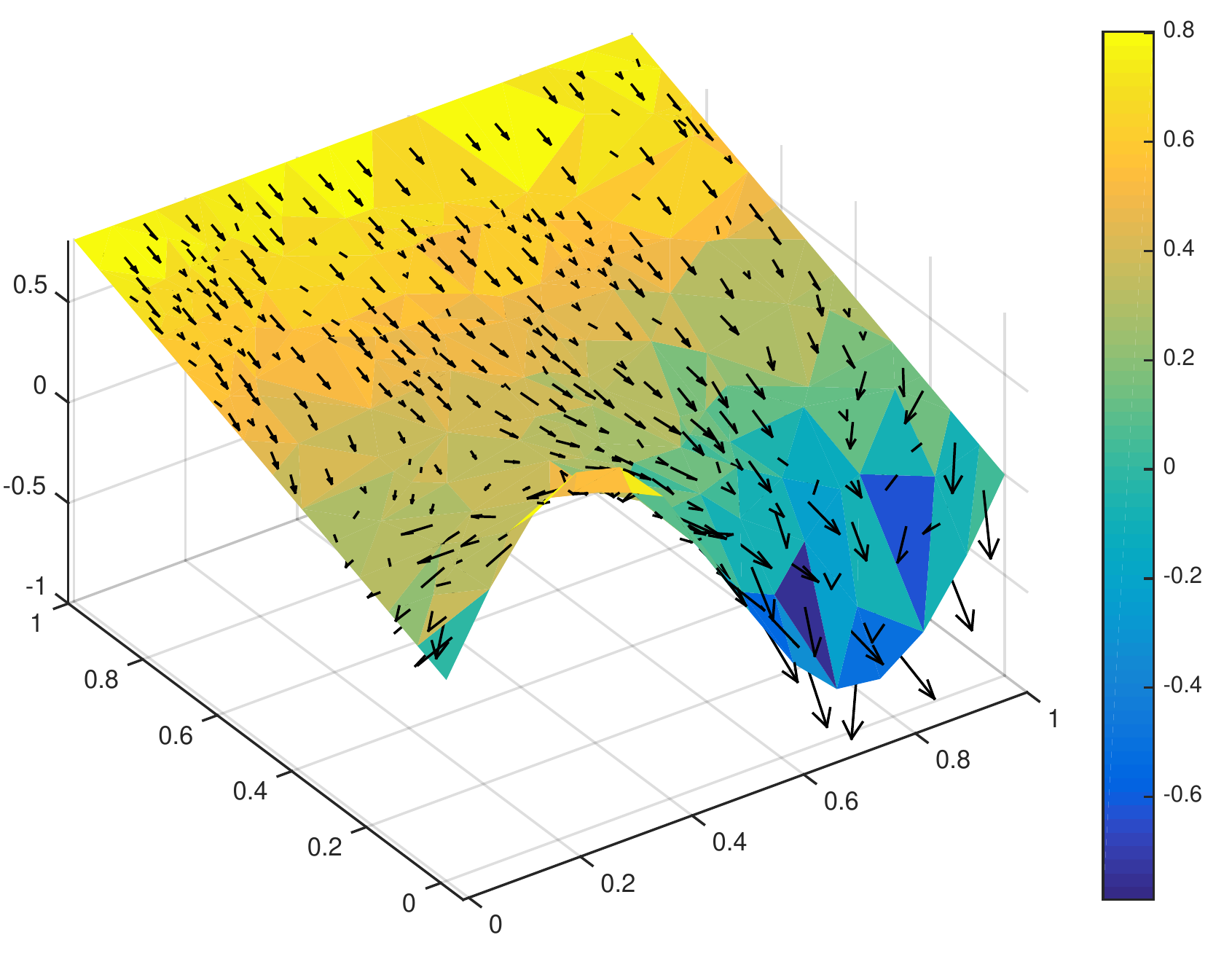} } }
        \end{subfigure}
        \qquad
        ~ 
          \begin{subfigure}[Flow Field gravity and pressure driven flow.]
                {\resizebox{6cm}{6.6cm}
{\includegraphics{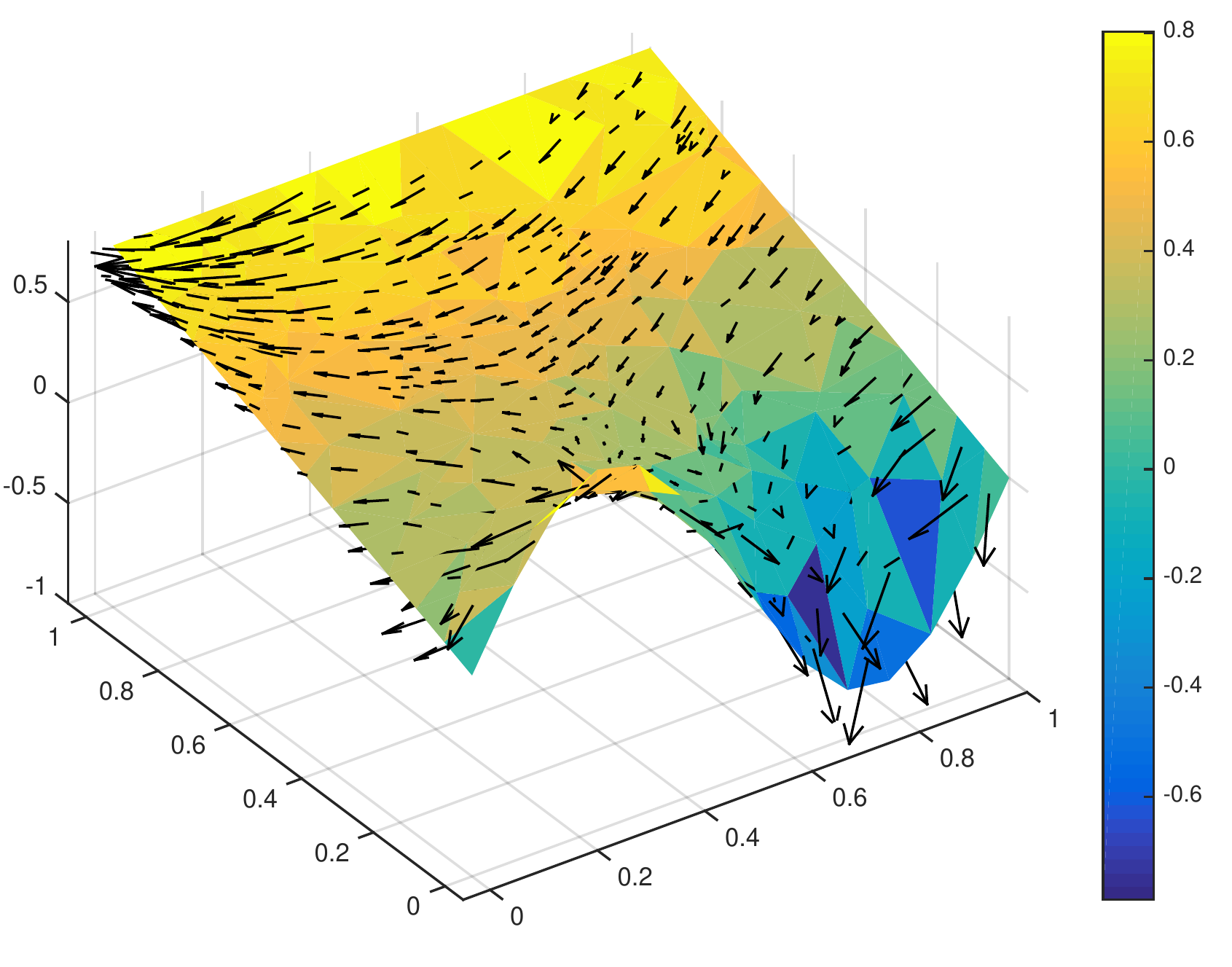} } }
        \end{subfigure}%
\\
        \begin{subfigure}[Expected boundary reaching times \newline gravity driven flow.]
                {\resizebox{6.5cm}{6.5cm}
{\includegraphics{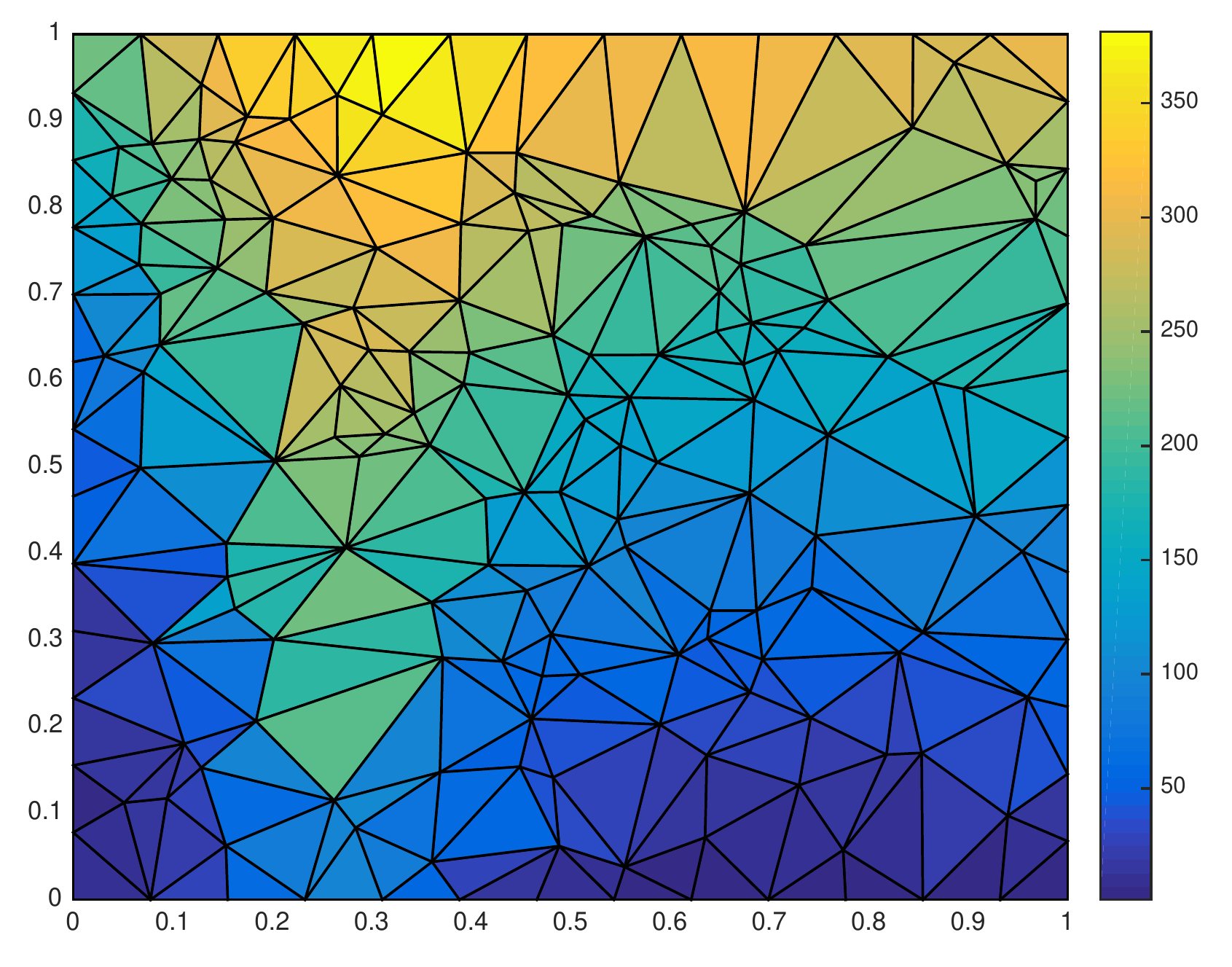} } }
        \end{subfigure}
        \qquad
        ~ 
          \begin{subfigure}[Expected boundary reaching times \newline gravity and pressure driven flow.]
                {\resizebox{6.5cm}{6.5cm}
{\includegraphics{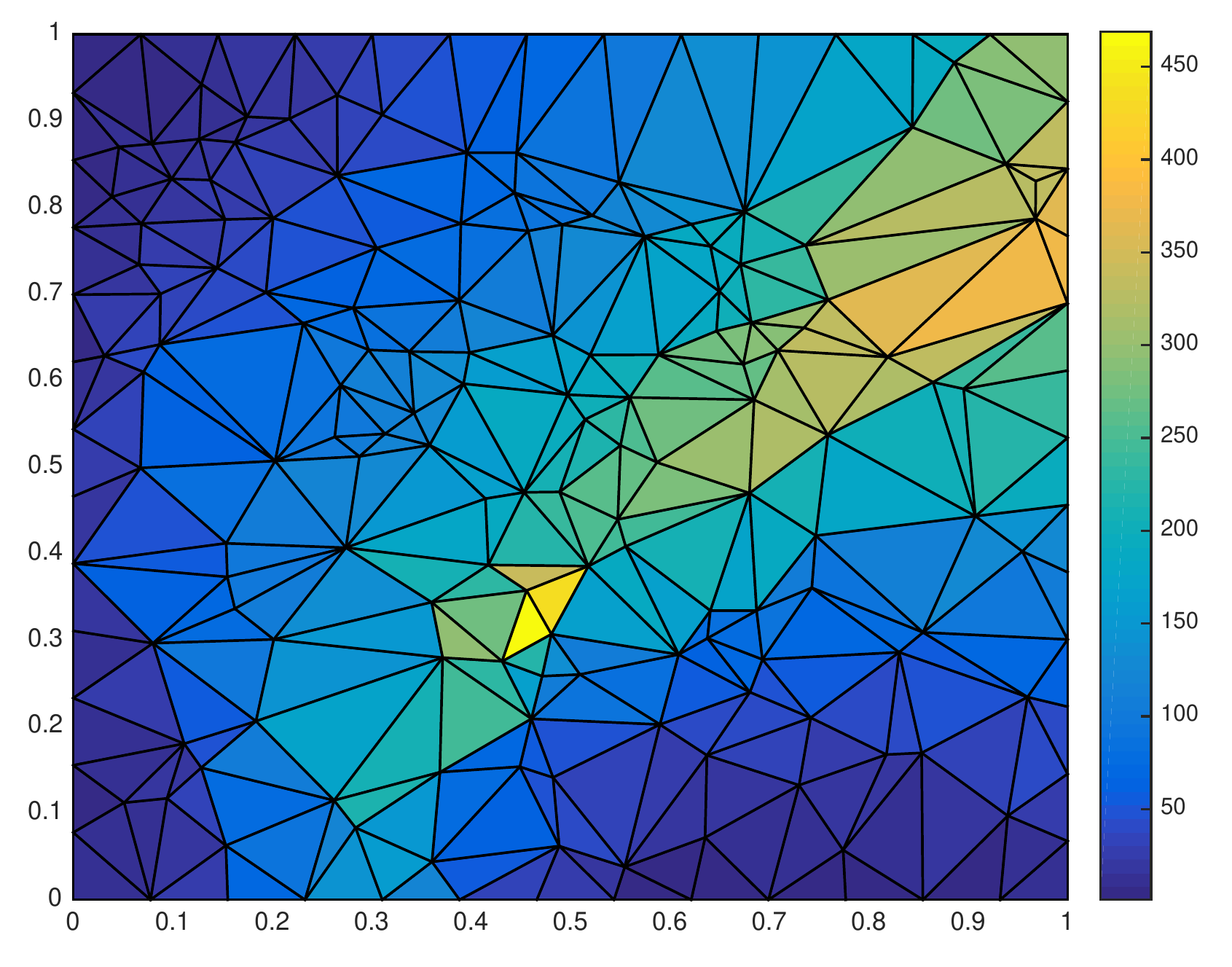} } }
        \end{subfigure}%
        \caption{First example for verification purposes}
\label{Fig Surface and Time Expected Ex 1} 
\end{figure}
\end{example}
\begin{example}\label{Ex Exploration Example}
The second example explores a more complex scenario. The surface is given by the function $\zeta: G\rightarrow \R$, $\zeta(x, y) = 0.5 \, x \,\sin(2\pi\, x) + 0.5\, x + 0.075\, \sin(6\pi \,y)$. Clearly, this function has significant more curvature and surface than the previous one, aside from not been harmonic. Again, for the first experiment we set $P \equiv 0$ and for the second, a pressure potential $P = 4000 \log ((x - 110)^2 + (y - 110)^2 )$ i.e., an extraction well beyond the upper right corner of the graphic. The well is placed near the highest point of the surface to counteract the dominant effect of the gravity as it sees in Figures \ref{Fig Surface and Time Expected Ex 2} (a) and (b). For the gravity driven flow the velocity field (see Figure \ref{Fig Surface and Time Expected Ex 2} (c)) attains the highest velocities and consequently the highest discharge rates in the valleys of the surface as expected. In the second case, due to the presence of the well the flow preference is split in two portions, one fraction moves towards the well, the other through the lower valleys and, the highest velocities are attained near the main evacuation points, see Figure \ref{Fig Surface and Time Expected Ex 2} (d). Finally, the expected reaching times agree with the intuitive notion, in the first case the highest values are in the top peaks and the lowest values near the boundary and the lower valleys, see Figure \ref{Fig Surface and Time Expected Ex 2} (e). In the second case, the lowest expected reaching times occur near the major evacuation points the lowest take place on the peaks away from both evacuation points,  see Figure \ref{Fig Surface and Time Expected Ex 2} (f). 
\begin{table}[h!]
\def\arraystretch{1.4}
\begin{center}
\begin{tabular}{ c c c c c c c}
    \hline
Experiment  & 
$\m[\u(\v)]$ &
$U_{\curv}$ & 
$U_{\fric} $ & 
$U_{\grav}$  &
Energy Balance &
$\m\big[\mathbb{E}(\tau | X(0) = K)\big]$ \\
           & 
[cm/s]  &
[erg/s] & 
[erg/s] & 
[erg/s] &
[erg/s] &
[s] \\[2pt]
\toprule
\shortstack{Gravity\\driven only} & 
$\begin{pmatrix}
    0.0086 \\
   -0.2650 \\
   -0.3012 \\
\end{pmatrix}$  &
   94.70545 &
   214.7014 &
  -30869.39 &
  -30559.98 &
198
\\[20pt] 
\shortstack{Gravity and \\pressure driven} & 
$\begin{pmatrix}
   0.3868 \\
   0.2218 \\
  -0.1802 \\
\end{pmatrix}$  &
   273.5967 &
   531.7068 &
  -16713.18 &
  -15907.87 &
177
\\ 
    \hline
\end{tabular}
\end{center}
\caption{Example \ref{Ex Exploration Example}: Energy Dissipation Table, 1046 Elements.}
\label{Tb Example 2}
\end{table}

We summarize in the Table \ref{Tb Example 2} below the overall behavior of the surface for both potentials.   
Again it follows that for both potentials, dominant effect is the gravity. In the f the order of scales states that $U_{\curv} \sim U_{\fric} \ll U_{\grav}$. In the second case the well pulls the flow uphill towards the highest point consequently, the value of $U_{\grav}$ contracts around 45 percent, while $U_{\curv}$ and $U_{\grav}$ increase drastically around 300 and 100 percent respectively but preserving the same order of scales $U_{\curv} \sim U_{\fric} \ll U_{\grav}$. It also follows that the increase of these two effect come from the fact that the new stream lines travel along more curved and therefore longer paths. Finally, it is observed that the average expected reaching time decreases in only 10 percent from the first to the second case; although there are now two main evacuation points, the top point demands more travel time due to the increase in the length of the stream lines and the magnitudes of the flow field.
%
\begin{figure}
        \centering
        \begin{subfigure}[Forcing Potential Surface $ \rho\, \g \, \zeta$.]
                {\resizebox{6.5cm}{6.5cm}
{\includegraphics{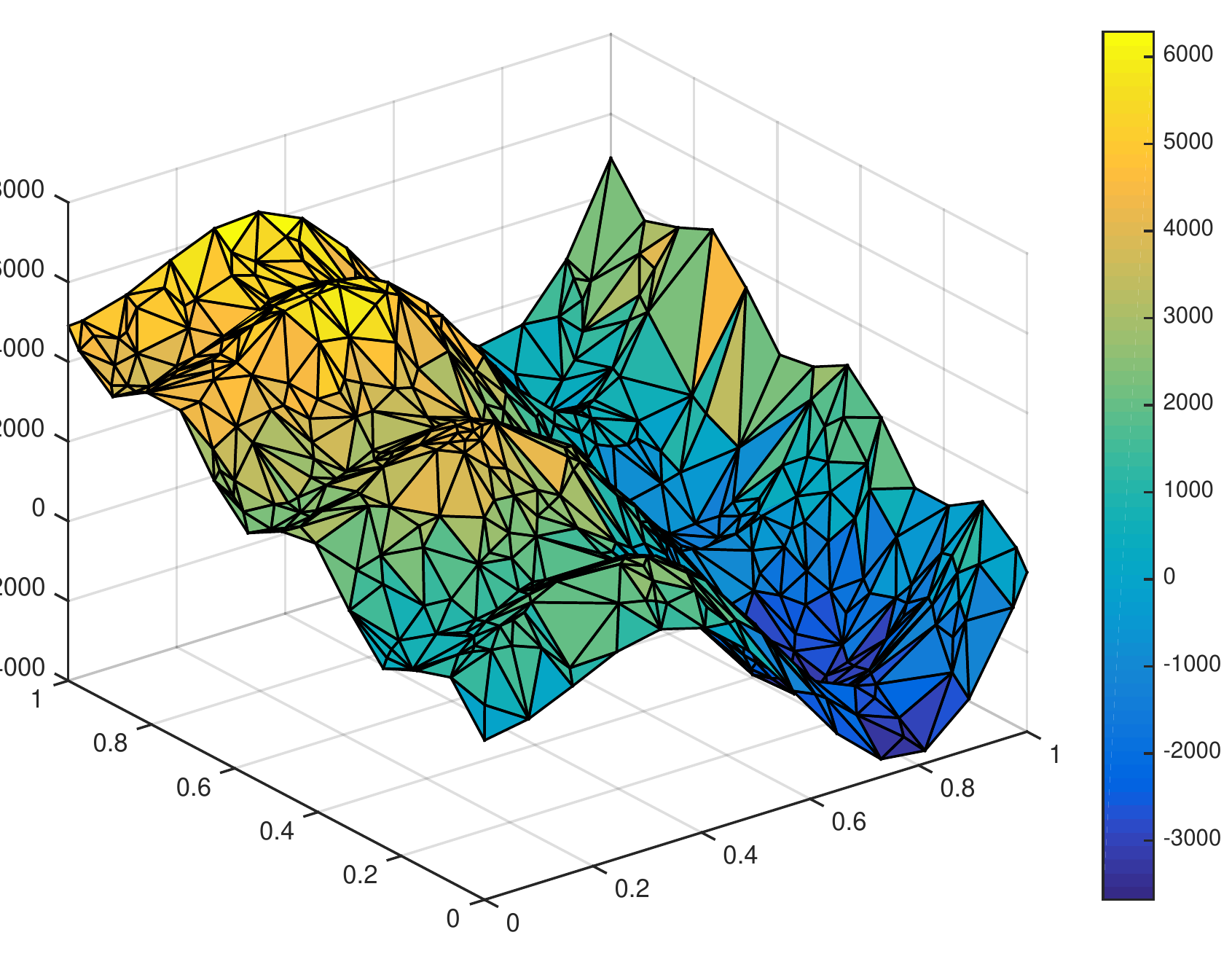} } }
        \end{subfigure}
        \qquad
        ~ 
          \begin{subfigure}[Forcing Potential Surface $\rho\, \g\, \zeta + P$.]
                {\resizebox{6.5cm}{6.5cm}
{\includegraphics{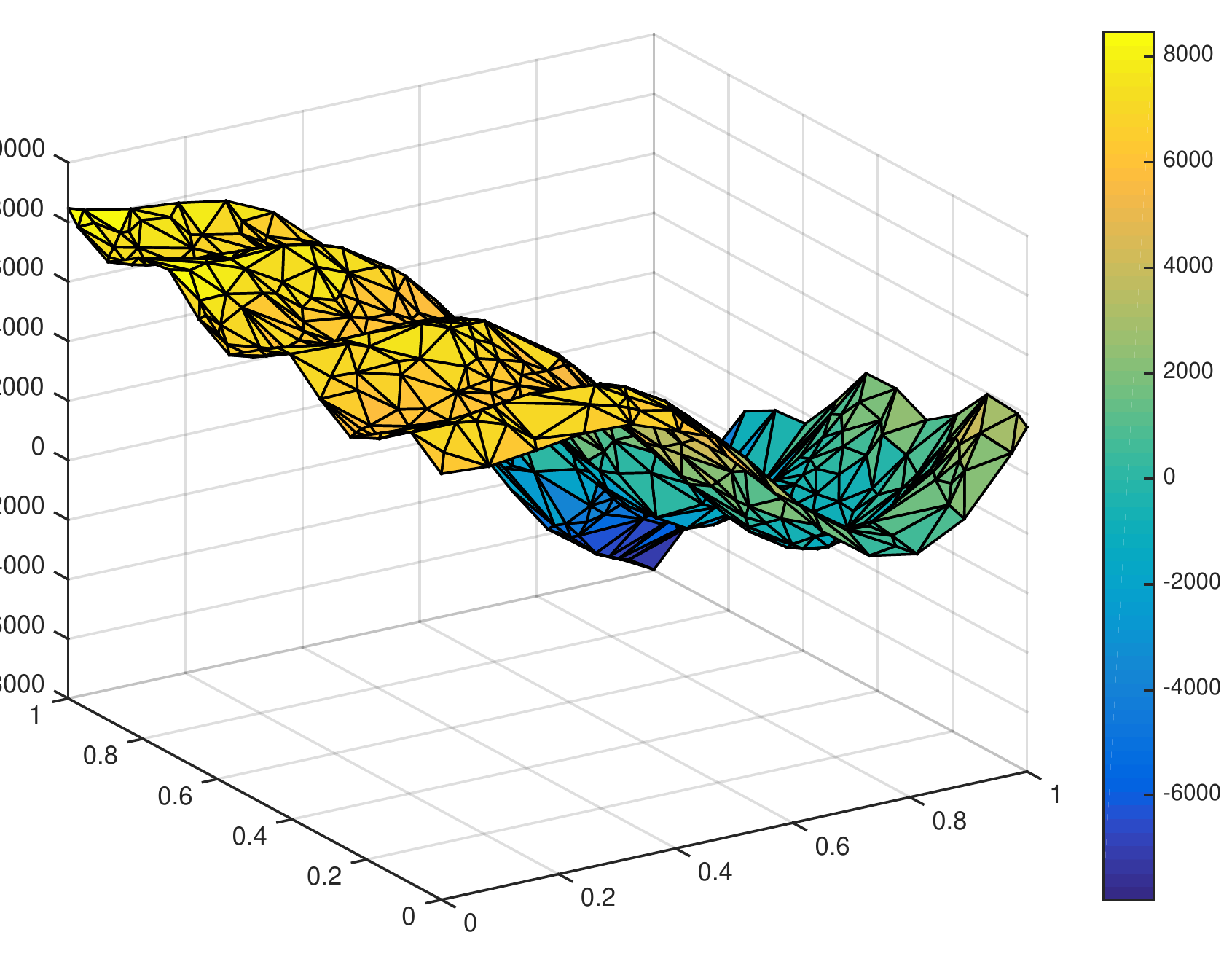} } }
        \end{subfigure}%
        \\
        \begin{subfigure}[Flow Field gravity driven flow.]
                {\resizebox{6.5cm}{6.5cm}
{\includegraphics{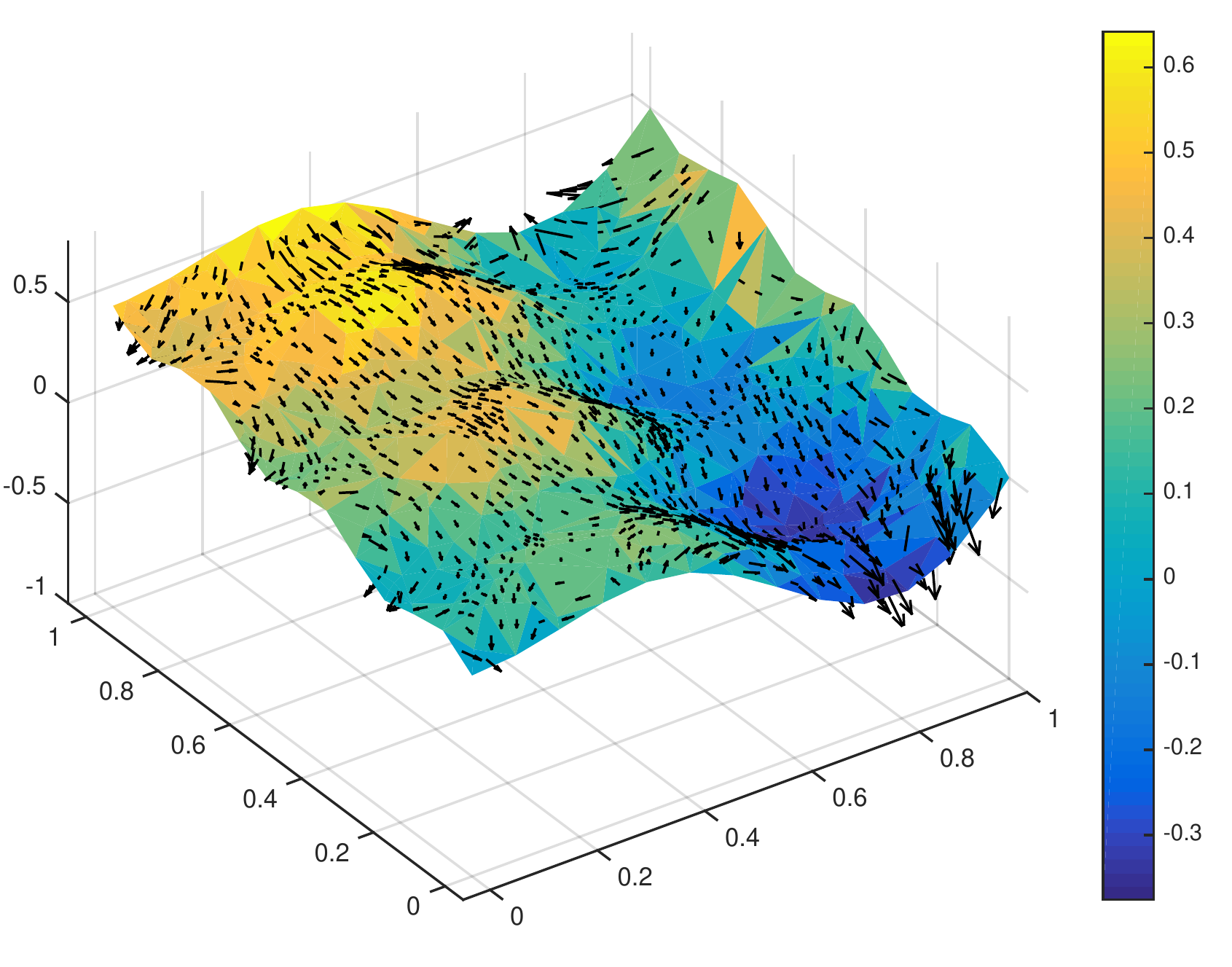} } }
        \end{subfigure}
        \qquad
        ~ 
          \begin{subfigure}[Flow Field gravity and pressure driven flow.]
                {\resizebox{6.5cm}{6.5cm}
{\includegraphics{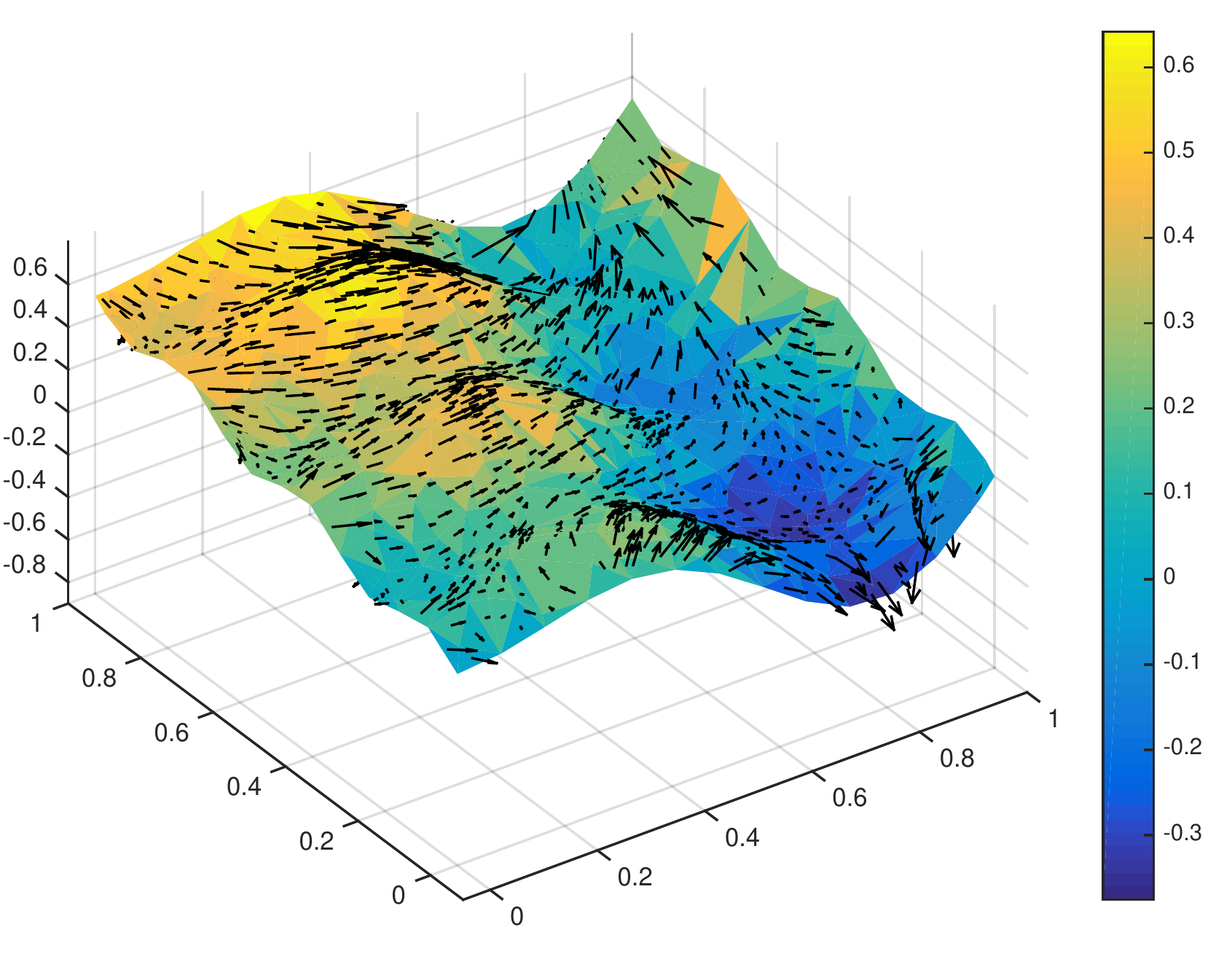} } }
        \end{subfigure}%
        \\
        \begin{subfigure}[Expected boundary reaching times \newline gravity driven flow.]
                {\resizebox{6.5cm}{6.5cm}
{\includegraphics{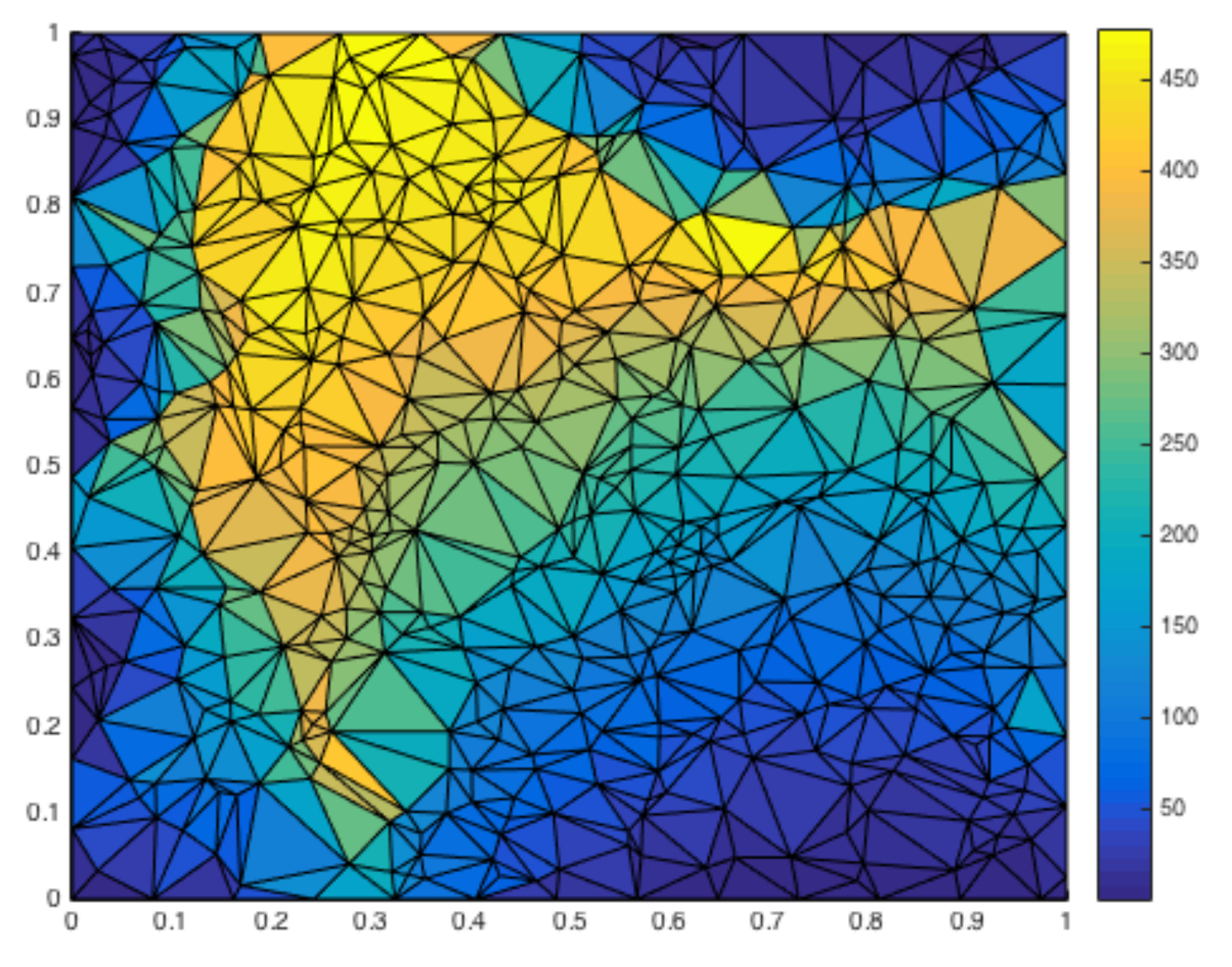} } }
        \end{subfigure}
        \qquad
        ~ 
          \begin{subfigure}[Expected boundary reaching times \newline gravity and pressure driven flow.]
                {\resizebox{6.5cm}{6.5cm}
{\includegraphics{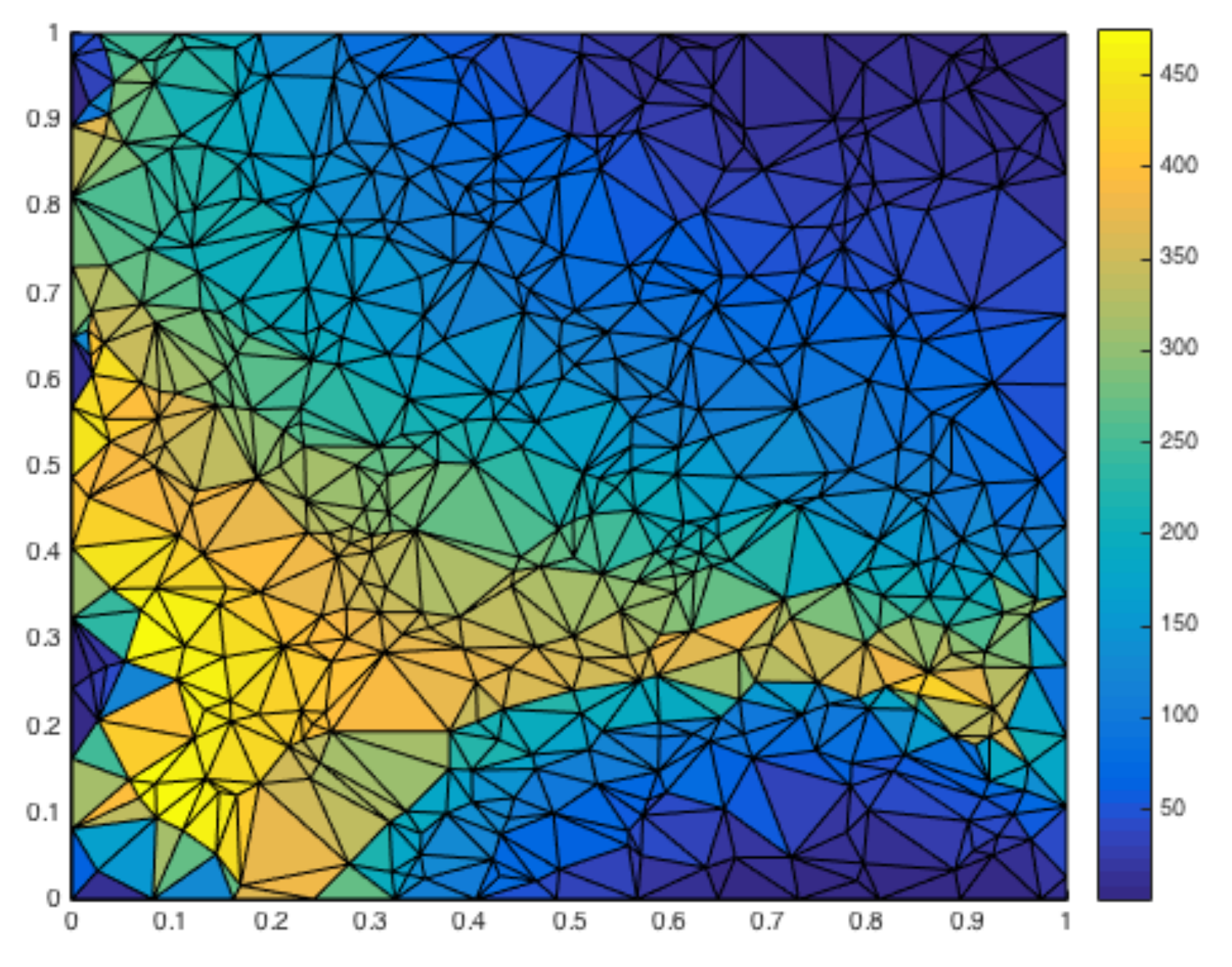} } }
        \end{subfigure}%
\caption{Second example for experimental purposes}
\label{Fig Surface and Time Expected Ex 2}
\end{figure}
\end{example}
%
%
%
%
%
%
%
%
%
\section{Conclusions}
%
%
The present work yields several conclusions listed below
\begin{enumerate}[(i)]
\item We proposed a method aimed to describe the flow on the fissures embedded in a porous medium and also to measure its impact in terms of energy dissipation and transport phenomena.

\item Unlike previous achievements, the present work is capable of describing explicitly the flow field on the fissure under mild hypotheses. Moreover, here the geometry of the manifold plays a central role for the quantifications we seek. 

\item Our approach is consistent and applicable to real case scenarios as pointed out in the introduction, because the available data for geological structures are discrete, therefore and interpolation process has to be done and the uncertainty coming from the knowledge of the geometry needs to be accounted for. 

\item The proposed method is remarkably simple as it is in essence a composition of linear operators and yet it computes accurately the analyzed phenomena and can be extended to other physical aspects. This simplicity is precisely its main fortitude because it makes it computationally efficient and thus fit for large scale computations, when it comes to simulate highly fractured media. 

\item The values of energy dissipation may not seem important in the presented experiments and they are of lower order scale with respect to the gravity however, at the field scale this effects become important as they add up the losses from a large number of extended fissures. In addition the average velocity for a given surface $\m[ \u(\v) ]$ defines what would the preferential flow direction of the medium would be, should there be only one embedded fissure. It follows that for a highly fractured medium the preferential flow direction (see \cite{FioriJankovic} for a different approach) is determined by the mean flow directions of each surface weighted vs. the corresponding discharge rate they host, relative to the total discharge rate along the cracks network. 

\item A natural question would be the extension of the method to a different interpolation method for the geometry, e.g. using splines rather than linear affine functions. However, this approach introduces a lot of computational complexity both on the geometric interpolation itself and in the construction of a conservative tangential flow analogous to the one furnished by the operator $F$ defined in Identity \eqref{Eq global velocity due curvature 3-D}. Hence, despite its mathematical interest, such approach is not viable for large scale simulations as the built-in complexity introduces issues such as 
ill-conditioned matrices, problems of numerical stability, poor quality of the numerical solutions and high computational costs.

\end{enumerate}
%
%
%
%
%
%
%
%
%
\section*{Acknowledgments}
%
%
This work was supported by the project HERMES 27798 from Universidad Nacional de Colombia, Sede Medell\'in.
%
%
%
%
%
%
%
%
%
%
%
%
%

%
%
%
%
%
%
%
%
%
\end{document}